\pdfoutput=1
\documentclass{amsart}
\usepackage{changepage} % <-- for adjustwidth

\usepackage{empheq}
\usepackage{color}
\usepackage[colorlinks,linkcolor=blue,anchorcolor=blue,citecolor=blue]{hyperref}
\usepackage{amssymb,amstext, amsbsy, amscd}
\usepackage{bbm}
\usepackage[mathscr]{eucal}
\usepackage{times}
\usepackage{amsmath, amsthm, amsfonts,extarrows}
\usepackage{multirow}
\usepackage{booktabs}
\usepackage[all]{xy}
\usepackage{upgreek}%for \updelta
\usepackage{subcaption}%for \subcaption in figure

\usepackage{tikz}
\usepackage{amsxtra}
\usepackage[framemethod=TikZ]{mdframed}

\usepackage{enumitem}% for defining cases

\allowdisplaybreaks[4]
\newtheorem{thm}{Theorem}[section]
\newtheorem{prop}[thm]{Proposition}

\newtheorem{theo}[thm]{Theorem}
\newtheorem{lemma}[thm]{Lemma}
\newtheorem{lemm}[thm]{Lemma}
\newtheorem{coro}[thm]{Corollary}

\newtheorem{prop-defn}[thm]{Proposition/Definition}

\theoremstyle{definition}
\newtheorem{definition}[thm]{Definition}

\newtheorem{exam}[thm]{Example}

\newcommand{\greeklabel}[1]{%
  \ifcase#1\or $\alpha$\or $\beta$\or $\gamma$\or $\delta$\fi}

\newmdenv[
    frametitle={A strategy to prove Gamma II(via Gamma I)}, % This puts the word "Strategy" at the top
    linecolor=black,
    linewidth=1pt,
    innertopmargin=2pt,
    innerbottommargin=10pt,
    skipabove=\baselineskip,
    skipbelow=\baselineskip
]{strategy}

\theoremstyle{remark}
\newtheorem{rema}[thm]{Remark}
\newtheorem{remark}[thm]{Remark}

\newlist{casesp}{enumerate}{3} %% new list environment based on enumerate with a max depth of 3
\setlist[casesp]{align=left, %% alignment of labels
                 listparindent=\parindent, %% same indentation as in normal text
                 parsep=\parskip, %% same parskip as in normal text
                 font=\normalfont\bfseries, %% font used for labels
                 leftmargin=0pt, %% total amount by which text is indented
                 labelwidth=0pt, %% width of labels (=how much they stick out on the left because align=left)
                 itemindent=.4em,labelsep=.4em, %% space between label and text
                 partopsep=0pt, %% extra vertical space above and below if separate paragraph
                 }
\setlist[casesp,1]{label=Case~\Roman*:,ref=\Roman*}
\setlist[casesp,2]{label=Case~\thecasespi.\arabic*:,ref=\thecasespi.\arabic*}
\setlist[casesp,3]{label=Case~\thecasespii.\alph*:,ref=\thecasespii.\alph*}

\newcommand{\bC}{{\mathbb C}}

\newcommand{\bP}{{\mathbb P}}

\newcommand{\bR}{{\mathbb R}}
\newcommand{\bZ}{{\mathbb Z}}

\newcommand{\cO}{{\mathcal O}}

  \newcommand{\<}{\langle}
  \renewcommand{\>}{\rangle}
  \newcommand{\Eff}{\mathrm{Eff}}
\newcommand{\suml}{\sum\limits}
\newcommand{\prodl}{\prod\limits}
\newcommand{\rme}{\mathrm{e}}
\newcommand{\bfd}{{\mathbf{d}}}

\newcommand{\bfq}{{\mathbf{q}}}
\newcommand{\bft}{\mathbf{t}}

\newcommand{\Hom}{\operatorname{Hom}}

\newcommand{\End}{\operatorname{End}}

\newcommand{\Spec}{\operatorname{Spec}}

\renewcommand{\Re}{\operatorname{Re}}
\renewcommand{\Im}{\operatorname{Im}}

\newcommand{\ch}{\operatorname{ch}}

\theoremstyle{plain}

\newtheorem*{GammaII}{Gamma conjecture II}
\newtheorem*{sthm}{Strategy-Theorem}

\usepackage{comment}

\makeindex

\numberwithin{equation}{section}

\title{Gamma conjecture II  via global Gamma-I}

\author{Jianxun Hu }
\address{School of Mathematics, Sun Yat-sen University, Guangzhou 510275, P.R. China}
\email{stsjxhu@mail.sysu.edu.cn}
\thanks{ %The first author  is supported in part by  research grant from the .
 }

\author{Hua-Zhong Ke}
\address{School of Mathematics, Sun Yat-sen University, Guangzhou 510275, P.R. China}
\email{kehuazh@mail.sysu.edu.cn}
\thanks{  
 }

\author{Changzheng Li}
 \address{School of Mathematics, Sun Yat-sen University, Guangzhou 510275, P.R. China}
\email{lichangzh@mail.sysu.edu.cn}

\author[Zhitong Su]{Zhitong Su}
\address{MOE-LCSM, School of Mathematics and Statistics, Hunan Normal University, Changsha 410081, P. R. China}
\email{suzht@hunnu.edu.cn}
\thanks{ %The first author  is supported in part by
 }

\begin{document}

\begin{abstract}
     For a Fano manifold $X$, Gamma conjecture II  aims to use $\mathcal{D}^b_{\rm{coh}}(X)$ to describe the asymptotic behavior of its Dubrovin connection  via $\widehat{\Gamma}$-integral structure. It was proposed by Galkin, Golyshev and Iritani, and can be regarded as  a quantitative refinement of Dubrovin’s conjecture on Fano manifolds with semisimple big quantum cohomology. As a step toward  Gamma conjecture II,  we define the Gamma‑I property at points satisfying the (SR) condition, arising from the original Gamma conjecture I. We prove that this property holds globally in the following sense: if it holds at one such point, then it holds throughout the connected component of the (SR)-region containing that point. Based on this global Gamma-I property, we establish a strategy-type theorem relating Gamma conjecture II to the Gamma-I property at a possibly non-semisimple point, together with an analysis of   small quantum cohomology. We further apply this theorem to prove Gamma conjecture II for del Pezzo surfaces; the proof combines Iritani's Galois action with additional elementary operations on exceptional collections, and   its most technically involved step consists in verifying the required global Gamma-I property.
\end{abstract}
\maketitle
\vspace{-2em}
\tableofcontents

\section{Introduction}
The structural correspondence between Gromov--Witten theory and classical algebraic geometry is a central theme in   symplectic geometry and complex geometry. This relationship has motivated a series of profound conjectures and the development of new categorical frameworks. In his 1998 ICM talk \cite{Dub98}, Dubrovin  proposed a  conjecture for Fano manifolds $X$, describing an  intriguing relationship between the big quantum cohomology $QH_{\rm big}(X)$ and the bounded derived category $\mathcal{D}^b_{\rm coh}(X)$ of coherent sheaves on $X$. Its first part asserts, qualitatively, that $QH_{\rm big}(X)$ is generically semisimple if and only if $\mathcal{D}^b_{\rm coh}(X)$ admits a full exceptional collection. Gamma conjecture II, proposed by Galkin, Golyshev, and Iritani \cite{GGI16}, can be regarded as a refinement of the quantitative part of Dubrovin's conjecture (see also Iritani's 2022 ICM talk \cite{Iri23}). It aims to describe  the Dubrovin connection associated with $QH_{\rm big}(X)$ in terms of $\mathcal{D}_{\rm coh}^b(X)$ via Iritani's Gamma-integral structure, when the latter admits a full exceptional collection.  An independent refinement of Dubrovin's conjecture was also proposed by Cotti, Dubrovin, and Guzzetti \cite{CDG24}.

To describe Gamma conjecture II, we consider the Dubrovin connection $\nabla$. Modulo convergence issues,
it is a meromorphic flat connection on the trivial $H^*(X)$-bundle over $H^*(X)\times\bP^1$ given by
\begin{align*}
    \nabla=\mathrm{d}+\suml_{i=0}^{s-1}{1\over z}(T_i\star_\bft)\mathrm{d}t_i-\left({1\over z}\left(E(\mathbf{t})\star_{\mathbf{t}}\right)-\mu\right){\mathrm{d}z\over z}.
\end{align*}
Here $z$ is the inhomogeneous coordinate on $\bP^1$,  $\{T_i\}_i$ is a homogeneous basis of $H^*(X)$, $\star_{\mathbf{t}}$ denotes the quantum product of $QH_{\rm big}(X)$, $E(\mathbf{t})$ is the Euler vector field, and 
 $\mu\in\End(H^*(X))$ is a grading operator.  
Denote by $K(X)$    the Grothendieck group  of topological complex vector bundles on $X$, and by $\mathcal{S}$   the space  of  $\nabla$-flat sections. The $\widehat{\Gamma}$-integral structure, introduced by Iritani \cite{Iri09}, is the image of a specified group homomorphism $$\mathcal{Z}^K: K(X) \to \mathcal{S},$$ defined  using the Gamma class $\widehat{\Gamma}_X$ \cite{Lib99, Lu08, KKP08, Iri09} of $X$,  a modification of Chern character, and a canonical $\mathbb{C}$-linear isomorphism $\mathcal{Z}: H^*(X)\to \mathcal{S}$ given by Givental's operator. 
   The Dubrovin connection $\nabla$ 
 has two singularities: 
 \begin{enumerate}
     \item[(i)] the regular singularity   at $z=\infty$, where the $\widehat{\Gamma}$-integral structure in $\mathcal{S}$ is involved; 
    \item[(ii)] the irregular singularity at $z=0$, where the Stokes structure in $\mathcal{S}$ is involved and
    characterized by the exponential growth of flat sections.
 \end{enumerate} 
 Gamma conjecture II asserts that these two structures are compatible in the following sense.  
 Assume that $QH_{\rm big}(X)$ is semisimple at   $\mathbf{t}\in H^*(X)$, so that $E\star_{\mathbf{t}}$ has eigenvalues $\{u_i\}_{i=1}^s$  with     corresponding normalized idempotent eigenvectors $\{\Psi_i\}_{i=1}^s\subset H^*(X)$ , where $s=\dim H^*(X)$. We  say that a flat section {\itshape $y\in \mathcal{S}$ respects $(u_i,\Psi_i)$ with the phase $\phi\in \mathbb{R}$ at  $\mathbf{t}$}, if   
\begin{equation*}
    e^{\frac{u_i(\bft)}{z}}y(\bft,z)\longrightarrow \Psi_i,\text{ as } z\rightarrow 0 \text{ along the sector } |\arg z-\phi| <\frac{\pi}{2}+\epsilon \text{ for some } \epsilon>0.
\end{equation*}
Such $y$ uniquely exists when $\phi$ is admissible at $\bft$, i.e. $\rme^{\mathbf{i}\phi}$ is not parallel to any nonzero difference $u_{i_1}(\bft)-u_{i_2}(\bft)$.

\begin{GammaII}[\protect{\cite[Conjecture 4.6.1]{GGI16}}] Assume  that $QH_{\rm big}(X)$ is convergent and semisimple at some $\mathbf{t}\in H^*(X)$ and that $\mathcal{D}^b_{\rm coh}(X)$ admits a   full exceptional collection.   Then for any phase $\phi\in\bR$ admissible at $\bft$, there exists a full exceptional collection $(E_1,\dots,E_s)$ of $\mathcal{D}^b_{\rm coh}(X)$, such that the flat section  $\mathcal{Z}^K(E_i)$ respects $(u_i,\Psi_i)$ at $\mathbf{t}$ with the phase $\phi$ for all $i$.
\end{GammaII} 
\noindent We refer to Section \ref{sec-Statement of Gamma conjecture II} for a more precise review of Gamma conjecture II. 

Currently, the only Fano manifolds for which Gamma conjecture II has been proved are the following: complex Grassmannians via quantum Satake by Galkin, Golyshev  and Iritani \cite{GGI16} (see also \cite{CDG24}), toric Fano manifolds via homological mirror symmetry by Fang and Zhou \cite{FZ19}, and quadrics via direct computations by Hu and Ke \cite{HK23}.\par

As noted in \cite[Section 4.6]{GGI16}, Gamma conjecture II refines part (3) of Dubrovin's conjecture, with the latter consisting of three parts. As recalled above, part (1) asserts that $QH_{\rm big}(X)$ is generically semisimple if and only if $\mathcal{D}^b_{\rm coh}(X)$ admits a full exceptional collection; this statement was made more precise in \cite{Bay04,HMT09}. Parts (2) and (3) concern, respectively, the Stokes matrix and the central connection matrix of the Dubrovin connection.

Besides Gamma conjecture II, Gamma conjecture I and the underlying Conjecture $\mathcal{O}$ were also proposed in \cite{GGI16}. It is a  fundamental fact that for a Fano manifold  $X$,   $\mathcal{D}^b_{\rm coh}(X)$   contains the structure sheaf $\mathcal{O}_X$ as an exceptional object. Gamma conjecture I concerns the asymptotic behavior near $z=0$ of the flat section    corresponding to $\mathcal{O}_X$ at $\mathbf{0}\in H^*(X)$. 
To be more precise, we consider the restriction of    $QH_{\rm big}(X)$ to the small quantum cohomology $QH_{\rm sm}(X)=(H^*(X)\otimes \mathbb{C}[e^{\mathbf{t}^{(2)}}], \star_{\mathbf{t}^{(2)}})$, where   $\mathbf{t}^{(2)}\in H^2(X)$. Accordingly, we have the small Dubrovin connection $\nabla^{\rm sm}$ on the trivial $H^*(X)$-bundle over $H^2(X)\times \mathbb{P}^1$,   the space $\mathcal{S}^{\rm sm}$ of $\nabla^{\rm sm}$-flat sections, as well as the $\widehat{\Gamma}$-integral structure $\mathcal{Z}^{K, \rm sm}: K(X)\to \mathcal{S}^{\rm sm}$. Consider the  linear operator $\hat c_1({\mathbf{t}^{(2)}})=c_1(X)\star_{\mathbf{t}^{(2)}}$ on $H^*(X)$. In that work, it was shown that whenever the spectral radius of $\hat c_1({\mathbf{0}})$ is a simple eigenvalue,  the vector subspace $\mathcal{A}(\mathbf{0})$ of  $\mathcal{S}^{\rm sm}$ with the smallest asymptotics as $z\to 0$ is one-dimensional. Gamma conjecture I asserts that $\mathcal{A}(\mathbf{0})=\mathbb{C}\widehat{\Gamma}_X$, or equivalently, $\mathcal{Z}^{K, \rm sm}(\mathcal{O}_X)(\mathbf{t}^{(2)}, z)|_{\mathbf{t}^{(2)}=\mathbf{0}}\in \mathcal{A}(\mathbf{0})$.

  Gamma conjecture I has been verified in a number of cases, including complex Grassmannians \cite{GGI16},  Fano threefolds of Picard rank one   \cite{GZ16},  Fano complete intersections in $\mathbb{P}^n$   \cite{GI19,SS20,Ke24}, del Pezzo surfaces  \cite{HKLY21}, the blow-up of $\mathbb{P}^n$ along $\mathbb{P}^r$  \cite{Yan22},    and   flag varieties   \cite{Cho25}. 
 In a recent collaboration  \cite{GHIKLS24} between S. Galkin,  H. Iritani and the present authors, counterexamples to the original Gamma conjecture I  were discovered among toric Fano manifolds. Moreover, modified formulations were provided and   proved to hold for all toric Fano manifolds, and   Gamma conjecture I over the K\"ahler moduli  was analyzed for these counterexamples.

 The statement of Gamma conjecture I appears to be  less directly related to Gamma conjecture II, and the existence of   counterexamples might even seem to weaken its importance. Nevertheless,  one of the main insights of the present paper, inspired by the detailed investigation of the counterexamples in \cite{GHIKLS24}, is that Gamma conjecture I is  much more deeply connected with Gamma conjecture II (see the Strategy-Theorem below) than previously understood; moreover, the aforementioned counterexamples do  not affect this new perspective.

We say that $X$ satisfies condition $(\mathrm{SR})$ at $\mathbf{t}^{(2)}\in H^2(X)$, if the linear operator $\hat{c}_1(\bft^{(2)})$ on $H^*(X)$ has a simple rightmost eigenvalue, namely, a simple eigenvalue with strictly largest real part. If such an eigenvalue  exists,  then it  is unique, and is denoted by  $u^{\rm SR}(\bft^{(2)})$. By   \cite[Remark 3.1.9]{GGI16}, the subspace $\mathcal{A}(\bft^{(2)})$ of $\mathcal{S}^{\rm sm}$  that  consists of   $y(\bft^{(2)},z)$ with  $\rme^{{u^{\rm SR}(\bft^{(2)})\over z}}y(\bft^{(2)},z)$ having moderate growth as $z\to0$ along the positive real axis is one-dimensional. 
We then say that 
{\itshape $X$ satisfies property \textbf{Gamma-I} at $\bft^{(2)}\in H^2(X)$}, if   $X$ satisfies condition $(\mathrm{SR})$ at $\bft^{(2)}$ and  $\mathcal{Z}^{K,{\rm{sm}}}(\cO)\in\mathcal{A}(\bft^{(2)})$. Note that the validity of the original Gamma conjecture I is a special case when $X$ satisfies Gamma-I at $\mathbf{0}$, and that the aforementioned counterexample does satisfy Gamma-I at some nonzero $\bft^{(2)}\in H^2(X)$.   By the $(\mathrm{SR})$-region of $X$, we mean the subset of $H^2(X)$ consisting of points at which  $X$ satisfies condition $(\mathrm{SR})$.   As one of the main results, we prove the following theorem in \textbf{Theorem \ref{thm--gammaIoverdomain}}, which   partially answers the third question in \cite[Question 6.12]{GHIKLS24}.

\begin{theo}[Global Gamma-I]\label{thm-global gamma-I} Let  $U$ be a  domain inside  the  $(\mathrm{SR})$-region of $X$.  If $X$ satisfies property   Gamma-I  at some point of $U$, then it satisfies property   Gamma-I  at every point of $U$. 
\end{theo}
\noindent This leads to the following strategy-type theorem (i.e. \textbf{Theorem \ref{thm:strategy}}). Let $K_{ts}\subset H^2(X)$ be the set of  {tame semisimple} points in the K\"ahler moduli, i.e. for $\bft^{(2)}\in K_{ts}$, $\hat{c}_1(\bft^{(2)})$ has only simple eigenvalues. 
\begin{sthm}
    Let $X$ be a  Fano manifold. Assume the following conditions. 
    \begin{enumerate}
[leftmargin=0pt, itemindent=0.5em, labelsep=0.5em]
    \item The big quantum cohomology of $X$ is convergent near the large radius limit.
    \item There exists $\bft^{(2)}_0\in H^2(X)$ at which $X$ satisfies property Gamma-I.
       \item There exists a domain $U$ inside the $(\mathrm{SR})$-region of $X$ such that  $\bft^{(2)}_0\in U$ and $U\cap K_{ts}\neq \varnothing$.
        
    \item  Denote $V_1:=\cO$. There exists $\bft^{(2)}_1\in U\cap K_{ts}$ such that:
    \begin{enumerate}
    \item there exist   $V_2,\dots,V_{s} \in \mathcal{D}^b_{\rm coh}(X)$ such that each $V_i$ is obtained from $\cO$ via the Galois action along some path $\gamma_i\subset H^2(X)$ starting at the same point  $\bft^{(2)}_1$;
        \item there exists a full exceptional collection $(\tilde{V}_1, \dots, \tilde{V}_s)$ in $\mathcal{D}^b_{\rm coh}(X)$, obtained from    the set $\{V_1, \dots, V_s\}$ by ``elementary operations";
        \item  there exists an admissible phase $\phi$ at $\bft^{(2)}_1$ such that $\mathcal{Z}^{K, \mathrm{sm}}(\tilde{V}_i)$ respects the pair $(u_i, \Psi_i)$ at $\bft^{(2)}_1$ with phase $\phi$ for all $i$, where the $u_i$ are ordered so that $\Im(\rme^{-\mathbf{i}\phi}u_1(\bft^{(2)}_1))>\dots>\Im(\rme^{-\mathbf{i}\phi}u_{s}(\bft^{(2)}_1))$.
     \end{enumerate}
 
\end{enumerate}    
    Then Gamma conjecture II holds for $X$.
\end{sthm}

There are several notable features of the above theorem, whenever it is applicable.  \begin{enumerate}
[leftmargin=2.1em, itemindent=0em, labelsep=0.5em]
        \item[i)] The starting point $\mathbf{t}_0^{(2)}$ may be  a non-semisimple point. Indeed, for  the del Pezzo surfaces $X_r$   investigated below, the quantum cohomology is non-semisimple at the starting point $\mathbf{t}_0^{(2)}=\mathbf{0}$ when $5\leq r\leq 8$ \cite{BM04}.

    \item[ii)] The validity of property Gamma-I 
    relies only on the small quantum cohomology and  does not require convergence of the big quantum cohomology. In particular, both $\mathbf{t}_0^{(2)}$ and $\mathbf{t}_1^{(2)}$ may lie outside   the domain of convergence of the big quantum cohomology.

    \item[iii)] The assumption (4.a) provides a strategy for constructing as many exceptional objects as possible, using the Galois action established by Iritani \cite{Iri09}. The obtained objects are ``initial data'' for (4.b), where the ``elementary operations'' involved  require individual treatment, including    mutations, taking differences, and related procedures, and therefore cannot be  defined  uniformly in a precise way.  

   \item[iv)] Under the strategy-style assumptions (2), (3), (4.a) and (4.b), the Strategy-Theorem becomes a reconstruction theorem, rather than merely a strategy, once assumption (4.c) is imposed; this is what we prove  in Theorem \ref{thm:strategy}.  
   
\end{enumerate}

\begin{rema}
The  method of combining Gamma conjecture I with the Galois action  has been  used to study Gamma conjecture II in \cite{GGI16, Iri20}. The main distinctions between these earlier applications and our Strategy-Theorem are the features described above, beyond  the use of Galois action.  The heart of the first half of the above theorem is the  Global Gamma-I  theorem. It  guarantees that  one can move from a starting point $\mathbf{t}_0^{(2)}$  to a  possibly  distant endpoint $\mathbf{t}_1^{(2)}$, although  a major difficulty in applying the  theorem   lies in finding such  a path in the $(\mathrm{SR})$-region.  Motivated by  
\cite[Sections 6 and 7]{GHIKLS24}, the  endpoint
  may be chosen near a boundary region of the
      K\"ahler moduli space where the spectrum of
$c_1(X)\star_{\mathbf{t}^{(2)}}$ organizes itself  in a way that reflects the geometry of
$X$ and makes condition (4) in the Strategy-Theorem accessible. 

 \end{rema}

\begin{rema}
     Assumption (3) implicitly requires   the small quantum cohomology to be  generically semisimple, and the current strategy works primarily when $\mathcal{D}^b_{\rm coh}(X)$ is generated by line bundles.  The above theorem provides a transparent   approach to studying Gamma conjecture II in this setting. We anticipate that the Strategy-Theorem may admit   generalizations to  Fano manifolds with non-semisimple small quantum cohomology and to cases in which $\mathcal{D}^b_{\rm coh}(X)$ is not generated by line bundles, by investigating an analogue of  the Global Gamma-I property in a more general setting.
\end{rema}  
 
It is known that a 2-dimensional  Fano manifold is isomorphic to $\mathbb{P}^1\times \mathbb{P}^1$, $\mathbb{P}^2$, or a del Pezzo surface $X_r$ (i.e. the blow-up of $\mathbb{P}^2$ at $r$ general points) where $1\leq r\leq 8$.  The del Pezzo surface $X_r$ is toric if and only if $1\leq r\leq 3$. Since its proposal in \cite{GGI16}, Gamma conjecture II has remained open for (non-toric) del Pezzo surfaces for a decade, even though   the geometry of these low-dimensional cases is relatively well understood. 
As another highlight of this paper, 
 we prove the following theorem, by verifying  all   assumptions of the Strategy-Theorem in a precise manner. The proof involves   a study of Puiseux expansions of eigenvalues of $c_1(X)\star_{\mathbf{t}^{(2)}}$. 
 \begin{theo}\label{thm: del Pezzo}
    Gamma conjecture II holds for all del Pezzo surfaces.
\end{theo} 
 
Part (1) of Dubrovin's conjecture for del Pezzo surfaces $X_r$ was established in \cite{BM04}. As shown in  \cite[Proposition 4.6.6]{GGI16}, Gamma conjecture II implies parts (2) and (3) of Dubrovin's conjecture. Therefore, we obtain the following result, as an immediate corollary of the above theorem.
 \begin{coro}
      Dubrovin's conjecture  holds for all del Pezzo surfaces.
 \end{coro}
\noindent We remark that    part (2) of Dubrovin's conjecture was shown for $X_6$ \cite{Ued05}.  

    Gamma conjecture II can be formulated straightforwardly for smooth projective manifolds with semisimple big quantum cohomology. 
    Let $S_r$ be an $r$-fold blow-up of $\bP^2$, i.e., $S_r$ is obtained by $S_r\xrightarrow{b_r}S_{r-1}\xrightarrow{b_{r-1}}\cdots\xrightarrow{b_2}S_1\xrightarrow{b_1}\bP^2$, where each $b_i$ is a  blow-up at one point. The rational surfaces are not Fano in general, but they can be deformed to $X_r$ when $r\leq 8$. As an application of Theorem \ref{thm: del Pezzo}, we obtain
\begin{coro}
    Gamma conjecture II holds for $S_r$ with $1\le r\le 8$.
\end{coro}

 We anticipate that our Strategy-Theorem has broader applications, thanks to the   extensive existing literature on Gamma conjecture I. Indeed, combined with mirror symmetry, it can be applied to verify Gamma conjecture II for Milnor hypersurfaces  (i.e. smooth  degree-$(1, 1)$ hypersurfaces in $\mathbb{P}^n\times \mathbb{P}^m$) \cite{HKLS26}.

Let us briefly describe how to apply the Strategy-Theorem to prove Theorem \ref{thm: del Pezzo}. We need to show that conditions (1)-(4) are satisfied. Condition (1) follows from the convergence criterion of Iritani \cite[Corollary 5.9]{Iri07}. Condition (2) is satisfied by  choosing $\bft^{(2)}_0=\mathbf{0}$ and using the original Gamma conjecture I proved in \cite{HKLY21}. For condition (4), we choose the endpoint $\bft^{(2)}_1\in K_{ts}$ near the boundary of the K\"ahler moduli reflecting the blow-up structure of $X_r$ from $\bP^2$; heuristically, this corresponds to letting the volumes of the exceptional divisors tend to infinity. After applying Galois actions to find initial objects $V_1,\dots,V_s$ as in (4.1), we succeed in finding ``elementary operations" on these initial objects to produce a full exceptional collection $(\tilde{V}_1,\dots,\tilde{V}_s)$ as in (4.2), which satisfies condition (4.3). 

The most technically involved step in the proof of Theorem \ref{thm: del Pezzo} is to verify condition (3), i.e., to find a suitable domain $U$ inside the $(\mathrm{SR})$-region of $X$ containing both $\mathbf{t}_0^{(2)}$ and $\mathbf{t}_1^{(2)}$. This depends on the distribution of the eigenvalues of $c_1(X)\star_{\mathbf{t}^{(2)}}$. The Perron--Frobenius theorem for nonnegative matrices plays a key role, in analogy with the proof of Conjecture $\mathcal{O}$ for flag varieties  \cite{ChLi17}.
We achieve this by reducing to a positivity problem for a matrix representing $c_1(X_r)\star_{\mathbf{t}^{(2)}}$.  As $\mathbf{t}^{(2)}$ moves from the origin toward  the boundary corresponding to the blow-up structure of $X_r$ from $\mathbb{P}^2$, we choose a basis
$\mathbf{1}, aH-\sum_iE_i, E_1,\cdots,E_r, [pt]$,
with a suitable parameter $a>0$. To apply the Perron--Frobenius theorem, we need careful entrywise estimates in this basis along a specific path approaching the boundary; these are based on the two key observations \ref{observation alpha} and \ref{observation beta} made in Section 5.2. They produce the desired domain inside the $(\mathrm{SR})$-region, where the Perron--Frobenius theorem gives the required simple rightmost eigenvalue.

We point out that our proof of Gamma conjecture II requires a rather involved manipulation of eigenvalues. In particular, the finer estimates for del Pezzo surfaces goes beyond what is provided by the general results on the quantum spectrum of blow-ups of algebraic surfaces in \cite{GS25}.
Moreover, 
 the proof relies on remarkably little explicit information regarding the Gromov-Witten invariants of $X_r$, as   provided in \eqref{eq--invariants} and Lemma \ref{lemma of table}.

 \begin{rema}
     More generally, it is important to study the   quantum spectrum of the Euler vector field $E(\mathbf{t})$. A conjecture concerning the quantum spectrum for blow-ups was proposed by Kontsevich in   several talks.
 The general case was proved by Iritani \cite{Iri25}, and the surface case was also shown by Gyenge and Szab\'o  \cite{GS25}. The quantum spectrum plays an important role in the breakthrough work \cite{KKPY25} by Katzarkov, Kontsevich, Pantev and Yu, which resolves the long-standing   irrationality problem for very general cubic fourfolds.  
 \end{rema}

\begin{rema} 
    In the context of noncommutative Hodge theory \cite{KKP08}, Katzarkov, Kontsevich, and Pantev proposed a $\mathbb{Q}$-structure on $QH_{\rm big}(X)$, which can be viewed as the rational counterpart of Iritani's $\widehat{\Gamma}$-integral structure.
    Gamma conjecture II can  be understood as a compatibility between the Betti structure
and the Stokes structure, as discussed by Hertling and Sevenheck \cite{HS07} in the context of TERP structures. 
  Gamma conjecture II was   formulated under the assumption that $QH_{\rm big}(X)$ is generically semisimple. For the setting in which $QH_{\rm big}(X)$ is non-semisimple, Sanda and Shamoto \cite{SS20}   proposed a corresponding generalization, referred to as a Dubrovin-type conjecture.
\end{rema}

The paper is organized as follows. In Section 2, we review the precise statement of Gamma conjecture II. In Section 3, we prove the Global Gamma-I theorem and the Strategy-Theorem. In Section 4, we show that Gamma conjecture II holds for del Pezzo surfaces. Finally, in Section 5, we prove a technical proposition   needed in the proof of Theorem \ref{thm: del Pezzo}.

\subsection*{Acknowledgements} 
The authors would like to thank Hiroshi Iritani and Sergey Galkin for the relevant collaboration on Gamma conjecture I that inspired the present  work.  
The first three authors are  supported in part by the National Key R \&  D Program of China No.~2023YFA1009801.
  J. Hu is also supported in part by NSFC Grant 12531003.
 H.-Z.  Ke is also supported in part by NSFC Grant 12271532.

\section{Brief review of Gamma conjecture II}\label{sec--preliminary}

In this section, we review Gamma conjecture II proposed by Galkin, Golyshev and Iritani \cite{GGI16}. It was originally stated for Fano manifolds with   semisimple   big quantum cohomology, and 
can be straightforwardly generalized to   projective manifolds with the same property.  By \cite{HMT09}, any such projective manifold has vanishing odd cohomology. 
We therefore begin with a projective manifold $X$ satisfying  $H^*(X)=H^{\mathrm{even}}(X,\bC)$.

\subsection{Quantum cohomology and Dubrovin connection}\label{sec-Dubrovinconn} 

Let $\mathrm{Eff}(X)\subset H_2(X,\mathbbm{Z})$ be the semigroup of effective curve classes. Denote by $\overline{M}_{0,k}(X,\mathbf{d})$ the moduli space of $k$-pointed, genus-zero stable maps of degree $\mathbf{d}\in\mathrm{Eff}(X)$. The genus-zero Gromov-Witten invariants of $X$ are defined as
\[
\<\prodl_{i=1}^k\tau_{a_i}(\gamma_i)\>_{\mathbf{d}}^X:=\int_{[\overline{M}_{0,k}(X,\mathbf{d})]^{\mathrm{vir}}}\prodl_{i=1}^kc_1(L_i)^{a_i}\mathrm{ev}_i^*\gamma_i,
\]
where $\gamma_i\in H^*(X)$, $a_i\in\bZ_{\ge0}$, $L_i$ is the universal cotangent line bundle at the $i$-th marked point, $\overline{M}_{0,k}(X,\mathbf{d})\xrightarrow{\mathrm{ev}_i}X$ is the evaluation map at the $i$-th marked point, and $[\overline{M}_{0,k}(X,\mathbf{d})]^{\mathrm{vir}}$ is the \emph{virtual fundamental class}. %with $D=\dim_\bC X-3+c_1(\mathbf{d})+k$. 
They satisfy a number of axioms \cite[Chapter 7.3]{CK99}, such as
Degree Axiom (i.e., the dimension constraint) and Divisor Axiom.
Roughly speaking, $\<\gamma_1,\dots,\gamma_k\>_{\bfd}^X:=\<\prodl_{i=1}^k\tau_{0}(\gamma_i)\>_{\mathbf{d}}^X$ is the (virtual) number of degree-$\mathbf{d}$ rational curves in $X$ intersecting the Poincar\'e dual cycles of $\gamma_1,...,\gamma_k$. We refer the reader to    \cite{CK99} for general introductions of Gromov-Witten invariants.   

Let $\{T_i\}_{i=0}^{s-1}$ be a homogeneous basis of $H^*(X)$, such that $T_0=\mathbf{1}$ and $\{T_i\}_{i=1}^{s'}$ is a nef basis in $H^2(X,\bZ)$. Write a general element of $H^*(X)$ as $\mathbf{t}=\suml_{i=0}^{s-1}t_iT_i$. We    decompose  $\mathbf{t}=\mathbf{t}^{(2)}+\mathbf{t}'$, where $\mathbf{t}^{(2)}\in H^2(X)$ and $\mathbf{t}'\in\bigoplus_{p\ne1}H^{2p}(X)$. The potential function of primary invariants is defined by $\mathcal{F}_X(\mathbf{t}):=\suml_{\mathbf{d}\in\Eff(X)}\<\exp\left(\tau_0(\mathbf{t}')\right)\>_{\mathbf{d}}^X\cdot\rme^{\mathbf{t}^{(2)}(\mathbf{d})}$. By the nef assumption of $\{T_i\}_{i=1}^{s'}$, we see that $\mathcal{F}_X(\mathbf{t})\in\bC[[t_0,q_1,...,q_{s'},t_{s'+1},\dots,t_{s-1}]]$, with $q_i:=\rme^{t_i}$.
 The quantum product $\star_\mathbf{t}$ on $H^*(X)$ is defined by
\begin{align}\label{def-quantumproduct}
(T_i\star_\mathbf{t}T_j,T_k)=\partial_{t_i}\partial_{t_j}\partial_{t_k}\mathcal{F}_X(\mathbf{t}), 
\end{align}
where $(\cdot,\cdot)$ is the Poincar\'e pairing on $H^*(X)$. The family  $(H^*(X),\star_\bft)$  is called the big quantum cohomology of $X$. 

We impose  the following convergence assumption: 
 \begin{align}\label{assumption:conv}
 \mbox{the potential function }  \mathcal{F}_X(\mathbf{t}) \mbox{ \textbf{converges near the large radius limit}. }
 \end{align} That is, there exists $\delta>0$ such that if 
 $|t_0|, |\rme^{t_1}|,...,|\rme^{t_{s'}}|, |t_{s'+1}|,...,|t_{s-1}|<\delta$, then $\mathcal{F}_X(\mathbf{t})$  converges absolutely as a series in $t_0,\rme^{t_1},\dots,\rme^{t_{s'}},t_{s'+1},\cdots,t_{s-1}$.

\begin{remark}
    If $X$ is Fano and  $H^*(X)$ is generated by $H^2(X)$ (e.g. a del Pezzo surface),  then $\mathcal{F}_X(\mathbf{t})$   converges near the large radius limit by  \cite[Corollary 5.9]{Iri07}. 
\end{remark}

Let $B\subset H^*(X)$ be the domain of convergence of the potential function $\mathcal{F}_X(\mathbf{t})$. 
For any $\mathbf{t}\in B$, $(H^*(X),\star_{\mathbf{t}})$ is a commutative and associative   $\bC$-algebra  with identity $\mathbf{1}$; 
we say that $\bft$ is a \textbf{semisimple} point if $(H^*(X),\star_{\mathbf{t}})$ is semisimple, namely  if $(H^*(X),\star_{\mathbf{t}})$ is isomorphic to $\bC\oplus\cdots\oplus\bC$ as a $\bC$-algebra.  We further call $\mathbf{t}$    a \textbf{tame semisimple} point if $(E(\bft)\star_\bft)$ has only simple eigenvalues, where 
$$E(\mathbf{t}):=c_1(X)+\suml_{i=0}^{s-1}\left(1-{1\over2}\deg T_i\right)t_iT_i$$
is called the  \textbf{Euler vector field}. 
Let $B_{ss}\subset B$ be the subset of all semisimple points. We say that the big quantum cohomology of $X$ is semisimple if $B_{ss}\neq\varnothing$. In this case,    $B\setminus B_{ss}$ is a divisor in $B$.

The (big) \textbf{Dubrovin connection}  $\nabla$ \cite{Dub96,Dub98,Dub99}  is a meromorphic flat connection on the trivial $H^*(X)$-bundle over $B\times\bP^1$.  With respect to a frame of constant sections, it reads
\begin{align*}
    \nabla=\mathrm{d}+\suml_{i=0}^{s-1}{1\over z}(T_i\star_\bft)\mathrm{d}t_i-\left({1\over z}\left(E(\mathbf{t})\star_{\mathbf{t}}\right)-\mu\right){\mathrm{d}z\over z},
\end{align*}
where $z$ is the inhomogeneous coordinate on $\bP^1$, 
and $\mu\in\End(H^*(X))$ is the grading operator defined by $\mu(T_i)={1\over2}\left(\deg T_i-\dim_\bC X\right)T_i$. Let $\widetilde{\bC^\times}$ be the universal cover of $\bC^\times=\bP^1\setminus\{0,\infty\}$, and denote by  $\widetilde{\nabla}$   the pullback of $\nabla$ via $B\times\widetilde{\bC^\times}\to B\times\bC^\times$. Denote
$$\mathcal{S}:= \{\widetilde{\nabla}\mbox{-flat sections  of the trivial } H^*(X)\mbox{-bundle over } B\times\widetilde{\bC^\times}\}.$$
 For $\mathbf{t}\in B$, following Givental \cite[Corollary 6.2]{Giv96}, we define a formal   endomorphism $L(\mathbf{t},z)\in\End(H^*(X))[[{1\over z}]]$ such that the Poincar\'e pairing $(L(\mathbf{t},z)T_i,T_j)$   equals
\[
\left(\rme^{-{\mathbf{t}^{(2)}\over z}}T_i,T_j\right)+\suml_{m\ge0}\left({-1\over z}\right)^{m+1}\suml_{\mathbf{d}\in\Eff(X)}\<\tau_m\left(\rme^{-{\mathbf{t}^{(2)}\over z}}T_i\right)\exp\left(\tau_0(\mathbf{t}')\right)\tau_0(T_j)\>_{\mathbf{d}}^X\cdot\rme^{\mathbf{t}^{(2)}(\mathbf{d})}. 
\]
The \textbf{cohomology framing} of $\mathcal{S}$ is the $\bC$-linear isomorphism  
\begin{align*}
  \mathcal{Z}: &\,\,\, H^*(X)\xrightarrow{\cong}\mathcal{S};\,\, \alpha \mapsto \left((\mathbf{t},z)\mapsto L(\mathbf{t},z)z^{-\mu}z^{c_1(X)\cup}\alpha\right),
\end{align*}
where $z^{-\mu}z^{c_1(X)\cup}:=\rme^{\mu\log z}\rme^{\left(c_1(X)\cup\right)\log z}$. Let $K(X)$ be the Grothendieck group of topological complex vector bundles on $X$. The \textbf{$K$-group framing} of $\mathcal{S}$ is the group homomorphism  
$$\mathcal{Z}^K: K(X) \xrightarrow{\quad}\mathcal{S};\,\,
V \mapsto (2\pi)^{-\frac{1}{2}}\mathcal{Z}\left(\widehat{\Gamma}_X\cup \widetilde{\ch}(V)\right).$$
Here $\widetilde{\ch}$ denotes the modified  Chern character defined as $\widetilde{\ch}(V):=\sum_{p\ge0}(2\pi\mathbf{i})^p\ch_p(V)$.
The Gamma class  $\widehat{\Gamma}_X$   of $X$ is defined by  
\begin{equation}
    \widehat{\Gamma}_X:=\prod_{i=1}^n\Gamma(1+\delta_i)\in H^*(X),
\end{equation}
where $\delta_i$ are the Chern roots of $X$ and $\Gamma(x)$ denotes   Euler's Gamma function. Iritani's \textbf{$\widehat{\Gamma}$-integral structure} \cite{Iri09} is the image $\mathcal{S}_\bZ\subset\mathcal{S}$ of $\mathcal{Z}^K$, which is a lattice in $\mathcal{S}$ of rank $\dim_\bC\mathcal{S}$. 

\begin{remark}
Inspired by ``remarkable identities" of Hosono, Klemm, Theisen, and Yau \cite{HKTY95} in the context of mirror symmetry,
Libgober \cite{Lib99} introduced the (inverse) Gamma class. In \cite{Iri09}, Iritani introduced the $\widehat{\Gamma}$-integral structure, and showed that, when $X$ is toric Fano, the structure matches with the natural integral structure in the Landau--Ginzburg $B$-model of $X$. A similar rational structure for quantum cohomology was also proposed by Katzarkov--Kontsevich--Pantec \cite{KKP08} in their study of non-commutative Hodge structures. We refer the reader to \cite[Section 1.5]{GGI16} for a detailed account of references.
\end{remark}

The free abelian group $\mathcal{S}_{\bZ}$ is endowed with a $\bZ$-valued bilinear pairing $[\cdot,\cdot)$ defined as follows. For any $s_1,s_2\in\mathcal{S}$, it is known that the pairing $\left(s_1(\mathbf{t},\rme^{-\pi\mathbf{i}}z),s_2(\mathbf{t},z)\right)$ is a constant function on $B\times\widetilde{\bC^\times}$. This induces a $\bC$-bilinear function $\mathcal{S}\times\mathcal{S}\xrightarrow{[\cdot,\cdot)}\bC$ defined by $[s_1,s_2):=\left(s_1(\mathbf{t},\rme^{-\pi\mathbf{i}}z),s_2(\mathbf{t},z)\right)$, which is nondegenerate, but is not necessarily symmetric or anti-symmetric. When restricted to $\mathcal{S}_\bZ$, the pairing takes values in $\bZ$. Moreover, for any $V_1,V_2\in K(X)$, we have 
\begin{equation}\label{equ-resp pairing}
[\mathcal{Z}^{K}(V_1),\mathcal{Z}^{K}(V_2))=\int_X\ch(V_1^\vee)\ch(V_2)\mathrm{td}(X)=:\chi(V_1,V_2)\in\bZ.    
\end{equation}
We refer the reader to \cite{Iri09} for more discussions on the pseudolattice $(\mathcal{S}_\bZ,[\cdot,\cdot))$.

\subsection{Exceptional collections} 
We refer the reader to \cite{BP94} for further details of this subsection. An object $F$ in the bounded derived category $\mathcal{D}_{\rm{coh}}^b(X)$ of coherent sheaves on $X$  is called \textbf{exceptional} if $\Hom(F, F) = \mathbb{C}$ and $\Hom(F, F[j]) = 0$,  where $[j]$  refers to applying the shift functor $j$ times, $j\neq 0$. An ordered tuple $(F_1, \dots, F_s)$ of such exceptional objects forms an \textbf{exceptional collection} if $\Hom(F_{i_1}, F_{i_2}[j]) = 0$ for all integers $j$ whenever $i_1 > i_2$. Two such collections are considered   equivalent up to shift if they differ only by a sequence of integer shifts applied to their respective components.  An exceptional collection is called \textbf{full} if it generates the category $\mathcal{D}_{\rm{coh}}^b(X)$.\par

Every exceptional object $F$ induces \textbf{left} and \textbf{right mutation functors}, $L_F$ and $R_F$, which are defined for any object $F' \in \mathcal{D}_{\rm{coh}}^b(X)$ via the distinguished triangles 
\begin{align*}
L&_F F'[-1] \to \Hom^\bullet(F, F') \otimes F \to F' \to L_F F',\, \text{ and}\\ &R_F F' \to F' \to \Hom^\bullet(F', F)^* \otimes F \to R_F F'[1].    
\end{align*}
 The set of  exceptional collections of length $N$ in $\mathcal{D}_{\rm{coh}}^b(X)$ admits a natural action by the braid group $B_N$. The generator $\sigma_i$ acts by replacing the adjacent pair $(F_i, F_{i+1})$ with $(F_{i+1}, R_{F_{i+1}}F_i)$, while $\sigma_i^{-1}$ replaces it with $(L_{F_i}F_{i+1}, F_i)$. The full triangulated subcategory generated by the collection is invariant under these mutations.

 Throughout the present paper, we abuse notation by identifying an object of $\mathcal{D}^b_{\rm coh}(X)$ with its corresponding class in the Grothendieck group $K(X)$ of topological complex vector bundles on $X$. For derived objects $E,F$, the Euler pairing $\chi$ in \eqref{equ-resp pairing} satisfies $\chi(E, F) = \sum_i (-1)^i \dim \operatorname{Ext}^i(E, F)$ by Hirzebruch-Riemann-Roch. Whenever $\mathcal{D}_{\rm{coh}}^b(X)$ possesses a full exceptional collection $(F_1, \dots, F_s)$, the abelian group $K(X)$ is torsion-free and the classes  $F_1, \dots, F_s$  in $K(X)$ form a $\bZ$-basis for $K(X)$; moreover, we have $\chi(F_i, F_i) = 1$ and $\chi(F_j, F_i) = 0$ for $j > i$. Thus, any full exceptional collection in $\mathcal{D}_{\rm{coh}}^b(X)$ naturally determines an exceptional basis for the pseudolattice $(K(X),\chi)$.

\subsection{Statement of Gamma conjecture II}\label{sec-Statement of Gamma conjecture II}

Assume $B_{ss}\neq\varnothing$.
 For a sufficiently small domain $U\subset B_{ss}$, Dubrovin \cite[Theorem 3.1, Lemma 3.2]{Dub99} proved that  there exist holomorphic coordinates $\{u_i\}$ on $U$ (called \textbf{canonical coordinates}), such that for each $\bft\in U$, the vectors $\{\partial_{u_i}\bft\}$ form an \textbf{idempotent basis} of $(H^*(X),\star_{\mathbf{t}})$ (i.e. $\partial_{u_i}\bft\star_{\mathbf{t}}\partial_{u_j}\bft=\delta_{ij}\partial_{u_i}\bft$), and $E(\bft)\star_\bft\partial_{u_i}\bft=u_i(\bft)\cdot\partial_{u_i}\bft$.  Assume moreover that  the \textbf{normalized idempotents} $\Psi_i(\bft):=\left(\partial_{u_i}\bft,\partial_{u_i}\bft\right)^{-{1\over2}}\partial_{u_i}\bft$  are   well-defined holomorphic maps on $U$ (which holds, for instance, if $U$ is simply connected). Then we say that 
 \begin{align}
     U \mbox{ is \textbf{properly-chosen} with respect to } \{(u_i,\Psi_i)\}.
 \end{align}
The functions $u_{i}$ are unique up to ordering, and   $\Psi_i$ are defined up to sign.

\begin{definition}\label{def-respect}
    Let $U\subset B_{ss}$ be properly-chosen with respect to $\{(u_i,\Psi_i)\}$ and $\phi\in\bR$. For $y\in\mathcal{S}$, we say that $y$ \textbf{respects $(u_i,\Psi_i)$ over $U$ with phase $\phi$}, if there exists $\varepsilon>0$ such that $\rme^{{u_i(\bft)}\over z}y(\bft,z)\to\Psi_i(\bft)$, as $z\to0$ with $|\arg z-\phi|<{\pi\over 2}+\varepsilon$. We say that $y$ respects $(u_i,\Psi_i)$ \textbf{at} $\bft_0\in U$ with phase $\phi$, if there exists $\varepsilon>0$ such that $\rme^{{u_i(\bft_0)\over z}}y(\bft_0,z)\to\Psi_i(\bft_0)$, as $z\to0$ with $|\arg z-\phi|<{\pi\over 2}+\varepsilon$.
\end{definition}

For $\bft\in B$, we say that a phase $\phi\in\bR$ is \textbf{admissible} at $\bft$ if $\rme^{\mathbf{i}\phi}$ is not parallel to any nonzero difference of eigenvalues of $E(\bft)\star_\bft$. This is an open condition on $\bR$.

\begin{prop}\cite[Proposition 2.5.1]{GGI16}
Suppose that $U\subset B_{ss}$ is properly-chosen with respect to $\{(u_i,\Psi_i)\}$, $\bft_0\in U$ and $\phi\in\bR$ is an admissible phase at $\bft_0$. Then there exists a unique basis $\{y_i\}$ of $\mathcal{S}$ satisfying the following property: there exists an open neighborhood $U'\subset U$ of $\bft_0$ such that $y_i$ respects $(u_i,\Psi_i)$ over $U'$ with phase $\phi$ for each $i$.
\end{prop}
 \noindent The   basis $\{y_i\}$ of $\mathcal{S}$ in the above proposition is called the \textbf{asymptotically exponential fundamental solution} (AEFS for short) to the equation $\widetilde{\nabla} s=0$ near $\bft_0$ associated to the phase $\phi$ with respect to $\{\Psi_i\}$. It follows from \cite[Lemma 2.7]{HK23} that the above AEFS $\{y_i\}$ can be uniquely characterized as follows: for each $i$, $y_i$ respects $(u_i,\Psi_i)$ \textbf{at} $\bft_0$ with phase $\phi$.

Roughly speaking, Gamma conjecture II expects that an AEFS lies in $\mathcal{S}_\bZ$. The precise statement is as follows.

\begin{GammaII}\cite[Conjecture 4.6.1]{GGI16}
    Assume that: (i) the big quantum cohomology of $X$ is semisimple; (ii)  $\mathcal{D}^b_{\rm coh}(X)$ admits a full exceptional collection. Suppose that $U\subset B_{ss}$ is properly-chosen with respect to $\{(u_i,\Psi_i)\}$, $\bft_0\in U$ and $\phi\in\bR$ is an admissible phase at $\bft_0$, where the $u_i$ are ordered so that $\Im\left(\rme^{-\mathbf{i}\phi}u_1(\bft_0)\right)\ge...\ge\Im\left(\rme^{-\mathbf{i}\phi}u_{s}(\bft_0)\right)$.
    Then there exists a full exceptional collection $(E_1, \dots ,E_{s})$ of $\mathcal{D}^b_{\rm coh}(X)$ such that for each $i$, $\mathcal{Z}^K(E_i)$ respects $(u_i,\Psi_i)$ at $\bft_0$ with phase $\phi$.
\end{GammaII}
\noindent By \cite[Remark 4.6.3]{GGI16} and \cite[Remark 4.13]{GI19}, the validity of Gamma conjecture II does not depend on the choices of semisimple points and admissible phases.

\section{Gamma conjecture II via Gamma-I over K\"{a}hler moduli}\label{sec-GammaI}

 Gamma conjecture I  was proposed for Fano manifolds in terms of the small quantum cohomology at the specialization of $\mathbf{q}=\mathbf{1}$ \cite{GGI16}. It appears to be largely independent of Gamma conjecture II. Nevertheless, 
in this section  we establish  the Strategy-Theorem (i.e. Theorem \ref{thm:strategy}) that reveals a deep connection between  Gamma conjecture II and property \textit{Gamma-I} over the K\"{a}hler moduli space. We now   further assume that $X$ is a Fano manifold (or more generally,  that $X$ can be deformed to a Fano manifold, by deformation invariance of Gromov-Witten theory).

\subsection{Small quantum cohomology}\label{sec-smallquantumcohomology}
We use notations in Section \ref{sec-Dubrovinconn}. Recall $q_k=e^{t_k}$ for $k=1,\dots,s'$. For any $\bft^{(2)}\in H^2(X)$, it follows from the dimension axiom of Gromov-Witten theory and the Fano property of $X$ that  
    \[
    T_i\star_{\bft^{(2)}} T_j \in\bC[q_1,\dots, q_{s'}]\otimes H^*(X).
    \]
   Therefore, we obtain a $\bC[\mathbf{q}]$-family of $\bC$-algebras $(H^*(X),\star_{\bft^{(2)}})$, called the small quantum cohomology of $X$. 
   
   Analogously to the big quantum cohomology case, the (small) Dubrovin connection 
    $\nabla^{\rm{sm}}$  is  a meromorphic flat connection on the trivial $H^*(X)$-bundle over $H^2(X)\times\bP^1$, and with respect to a frame of constant sections, it reads $$\nabla^{\rm{sm}}=\mathrm{d}+\sum_{i=1}^{s'}{1\over z}(T_i\star_{\bft^{(2)}})\mathrm{d}t_i-\left({1\over z}\hat{c}_1(\bft^{(2)})-\mu\right){\mathrm{d}z\over z},\mbox{ where } \hat{c}_1(\bft^{(2)}):= c_1(X)\star_{\bft^{(2)}}.$$
    Moreover, we obtain the space $\mathcal{S}^{{\rm{sm}}}$ of $\widetilde{\nabla}^{\rm{sm}}$-flat sections, together with cohomology framing $\mathcal{Z}^{{\rm{sm}}}$, $K$-group framing $\mathcal{Z}^{K,{\rm{sm}}}$ and $\widehat{\Gamma}$-integral structure $\mathcal{S}_\bZ^{\rm{sm}}$, defined as in the big case with   the operator $L(\mathbf{t}, z)$ replaced by $L(\mathbf{t}^{(2)}, z)$.

In particular, for any $\bft^{(2)}\in B\cap H^2(X)$, we have $\mathcal{Z}^{K}(V)(\bft^{(2)},z)=\mathcal{Z}^{K,{\rm{sm}}}(V)(\bft^{(2)},z)$ for any $V\in K(X)$.

We will use Iritani's \textbf{Galois action} of $H^2(X,\bZ)$ on $\mathcal{S}_\bZ^{\rm{sm}}$, described as follows. For any $\xi\in H^2(X,\bZ)$, consider the isomorphism of trivial $H^*(X)$-bundles over $H^2(X)\times\bP^1$ defined by
\[
H^*(X)\times(H^2(X)\times\bP^1)\to H^*(X)\times(H^2(X)\times\bP^1);(\alpha,\bft^{(2)},z)\mapsto(\alpha,\bft^{(2)}-2\pi\mathbf{i}\xi,z).
\]
Then it is an automorphism of $\nabla^{\rm{sm}}$, and it induces an action on $\mathcal{S}_\bZ^{\rm{sm}}$
\begin{align}\label{eqn-Galoisaction}
G(\xi):\mathcal{S}_\bZ^{\rm{sm}}\to\mathcal{S}_\bZ^{\rm{sm}}; s\mapsto\left((\bft^{(2)},z)\mapsto s(\bft^{(2)}-2\pi\mathbf{i}\xi,z)\right).
\end{align}
To describe the action $G(\xi)$ in terms of $K$-group framing, let $L_\xi$ be the topological complex line bundle with $c_1(L_\xi)=\xi$, and we have
\begin{align}\label{equ-monodromy-invariance}
    \mathcal{Z}^{K,{\rm{sm}}}(V\otimes L_\xi)=G(\xi)\left(\mathcal{Z}^{K,\rm{sm}}(V)\right),\quad\forall V\in K(X).
\end{align}
We refer the reader to \cite{Iri09} for more details. For any path $\phi:[0,1]\to H^2(X)$ with $\phi(0)=\bft^{(2)}$ and $\phi(1)=\bft^{(2)}-2\pi\mathbf{i}\xi$, we will say that \textbf{$V\otimes L_\xi$ is obtained from $V$ via Galois action given by the path $\phi$}.

\subsection{Condition $(\mathrm{SR})$ and Gamma-I}\label{sec-condtion*gammaI}
In this subsection, we recall condition $(\mathrm{SR})$, introduce property Gamma-I, and prove the global Gamma-I Theorem \ref{thm--gammaIoverdomain}. We emphasize that the semisimplicity of the big/small quantum cohomology is \emph{NOT} assumed in this subsection.

By condition $(\mathrm{SR})$ at $\bft^{(2)}$, we mean 
\begin{align}
    \tag*{ \rm ({SR})} \hat{c}_1(\bft^{(2)}) \mbox{\,\, has a simple rightmost eigenvalue}.
\end{align}
That is,  the linear operator $\hat{c}_1(\bft^{(2)})\in\End(H^*(X))$ has a (unique) simple eigenvalue $u$, such that    $\Re u'<\Re u$ for every other eigenvalue $u'$ of $\hat{c}_1(\bft^{(2)})$. 
By the $(\mathrm{SR})$-region of $X$, we mean the subset  $\{\bft^{(2)}\in H^2(X) \mid X \mbox{ satisfies } (\mathrm{SR})\mbox{ at }\bft^{(2)} \}$ of $H^2(X)$. This is an open subset of $H^2(X)$, which need not   be  connected when nonempty.    

Suppose that $X$ satisfies condition $(\mathrm{SR})$  at ${\bft^{(2)}}\in H^2(X)$,  and let $u\in\bC$ be the simple rightmost eigenvalue of $\hat{c}_1(\bft^{(2)})$. Denote by  $\mathcal{A}(\bft^{(2)})$ the $\bC$-linear space consisting of $\widetilde{\nabla}^{\rm{sm}}$-flat sections $y(\bft^{(2)},z)$ such that $\rme^{{u\over z}}y(\bft^{(2)},z)$ has moderate growth as $z\to0$ along the positive real axis (i.e. $|\arg z|=0$). Then $\dim_\bC\mathcal{A}(\bft^{(2)})=1$ by \cite[Remark 3.1.9]{GGI16}.

\begin{definition}[Property Gamma-I]
We say that \textbf{$X$ satisfies Gamma-I at $\bft^{(2)}\in H^2(X)$}, if   $X$ satisfies condition $(\mathrm{SR})$ at $\bft^{(2)}$ and  $\mathcal{Z}^{K,{\rm{sm}}}(\cO)\in\mathcal{A}(\bft^{(2)})$.
\end{definition}

\noindent We point out that the property Gamma-I does \emph{NOT} assume convergence or semisimplicity of big quantum cohomology.

\begin{remark}
The original Gamma conjecture I \cite[Conjecture 3.4.3]{GGI16} is slightly stronger than the statement that, if a Fano manifold satisfies condition $(\mathrm{SR})$ at $\mathbf{0}\in H^2(X)$, then it satisfies Gamma-I at $\mathbf{0}$. Counter-examples were found in \cite{GHIKLS24}, for which  the aforementioned statement fails but   Gamma-I at some nonzero $\bft^{(2)}$ is satisfied.  
\end{remark}

Assume that $X$ satisfies $(\mathrm{SR})$ at $\bft_0^{(2)}$  with $u_0\in\bC$ the simple rightmost eigenvalue. The equation $\det\left(u-(\hat{c}_1(\bft^{(2)}))\right)=0$ determines, by the Implicit Function Theorem, a unique holomorphic function $u(\bft^{(2)})$ near $\bft^{(2)}_0$ with $u(\bft^{(2)}_0)=u_0$.  As a consequence, for any domain $U$ inside the $ (\mathrm{SR})$-region of $X$, there are unique holomorphic maps $U\xrightarrow{u}\bC$ and $U\xrightarrow{\psi}H^*(X)$, such that for each $\bft^{(2)}\in U$, $u(\bft^{(2)})$ is the simple rightmost eigenvalue of $\hat{c}_1(\bft^{(2)})$, and $\psi(\bft^{(2)})$ is the idempotent generator of the associated one-dimensional eigenspace. Whenever the   normalized idempotent   $\Psi:=(\psi,\psi)^{-{1\over 2}}\psi$ is a well-defined holomorphic map (which holds, for instance, if  $U$ is simply connected),  we say that
\begin{align}
    (u,\Psi) \mbox{ is \textbf{a simple rightmost pair} over }U.
\end{align}

Before we investigate Gamma-I, we need the following result from linear algebra. Note that with respect to the Poincar\'e pairing, $\hat{c}_1(\bft^{(2)})$ is symmetric, and the grading operator $\mu$ is anti-symmetric by degree-counting. 
\begin{lemma}\label{lemma-linearalgebra}
    Suppose that $V$ is a finite-dimensional $\bC$-linear space, and $(\cdot,\cdot)$ is a bilinear pairing on $V$ which is symmetric and non-degenerate. Let $S,A\in\End_\bC(V)$ be such that $S$ is symmetric and $A$ is anti-symmetric with respect to $(\cdot,\cdot)$. Suppose that $\lambda$ is a simple eigenvalue of $S$ with the corresponding eigenspace $V_\lambda$, and let $V'$ be the direct sum of generalized eigenspaces of $S$ with eigenvalues $\ne\lambda$. Then: (i) $V'=\Im(S-\lambda\cdot id)$; (ii) $V_\lambda$ and $V'$ are mutually orthogonal complements; (iii) $A(V_\lambda)\subset V'$.
\end{lemma}
\begin{proof}
    Statement (i) follows from the assumption $\dim_\bC V_\lambda=1$. The statement (ii) holds since generalized eigenspaces corresponding to distinct eigenvalues of $S$ are mutually orthogonal. For (iii), let $v\in V_\lambda\setminus\{0\}$;  noting that $A$ is anti-symmetric, we have
    \[
    (A(v),v)=-(v,A(v))\Rightarrow(A(v),v)=0\Rightarrow A(v)\perp v.
    \]
   Then $A(v)\perp V_\lambda$ since $\dim_\bC V_\lambda=1$, and hence $A(v)\in V'$ by (ii). This proves (iii).
\end{proof}

The following Lemma \ref{lemma-partialforuinsmallquantum} and \ref{lemma--localholomorphicflatsectionofsmalldualconn} concern a \textbf{simple pair} $(u,\Psi)$ over a domain $U\subset H^2(X)$, that is, holomorphic maps $U\xrightarrow{u}\bC$ and $U\xrightarrow{\Psi} H^*(X)$, such that for each $\bft^{(2)}\in U$, $u(\bft^{(2)})$ is a simple eigenvalue of $\hat{c}_1(\bft^{(2)})$, and $\Psi(\bft^{(2)})$ is a normalized idempotent in the associated one-dimensional eigenspace. A simple rightmost pair is a special case of a simple pair.

\begin{lemma}\label{lemma-partialforuinsmallquantum}
Suppose that  $(u,\Psi)$ is a simple pair over a domain $U\subset H^2(X)$. Then, for every $1\leq j\leq s'$, we have $T_j\star_{\bft^{(2)}}\Psi(\bft^{(2)})=\partial_{t_j}u(\bft^{(2)})\cdot\Psi(\bft^{(2)})$, $\forall\bft^{(2)}\in U$.
\end{lemma}
\begin{proof}
From the flatness of $\nabla^{\rm{sm}}$ and $\hat{c}_1(\bft^{(2)})\Psi(\bft^{(2)})=u(\bft^{(2)})\cdot\Psi(\bft^{(2)})$, we get
\begin{align*}
&\partial_{t_j}u(\bft^{(2)})\cdot\Psi(\bft^{(2)})+\left(u(\bft^{(2)})-\hat{c}_1(\bft^{(2)})\right)\partial_{t_j}\Psi(\bft^{(2)})\\
={}&T_j\star_{\bft^{(2)}}\Psi(\bft^{(2)})+[(T_j\star_{\bft^{(2)}}),\mu]\Psi(\bft^{(2)}).
%T_j\star_{\bft^{(2)}}\mu\left(\Psi(\bft^{(2)})\right)-\mu\left(T_j\star_{\bft^{(2)}}\Psi(\bft^{(2)})\right).
\end{align*}
Let $H'_{\bft^{(2)}}\subset H^*(X)$ be the direct sum of generalized eigenspaces of $\hat{c}_1(\bft^{(2)})$ with eigenvalues $\neq u(\bft^{(2)})$. Then $H'_{\bft^{(2)}}=\mathrm{Im}\left(u(\bft^{(2)})-\hat{c}_1(\bft^{(2)})\right)$ by Lemma \ref{lemma-linearalgebra} (i). Observe that $T_j\star_{\bft^{(2)}}\Psi(\bft^{(2)})\in\bC\Psi(\bft^{(2)})$,  and so it suffices to show that $[(T_j\star_{\bft^{(2)}}),\mu]\Psi(\bft^{(2)})\in H'_{\bft^{(2)}}$. 

Note that Lemma \ref{lemma-linearalgebra} (iii) implies  $\mu(\bC\Psi(\bft^{(2)}))\in H'_{\bft^{(2)}}$. So $\mu\left(T_j\star_{\bft^{(2)}}\Psi(\bft^{(2)})\right)\in H'_{\bft^{(2)}}$. As a consequence, we have
\begin{align*}
    0&=\left(\mu\left(T_j\star_{\bft^{(2)}}\Psi(\bft^{(2)})\right),\Psi(\bft^{(2)})\right)=-\left(T_j\star_{\bft^{(2)}}\Psi(\bft^{(2)}),\mu\left(\Psi(\bft^{(2)})\right)\right)\\
    &=-\left(\Psi(\bft^{(2)}),T_j\star_{\bft^{(2)}}\mu\left(\Psi(\bft^{(2)})\right)\right).
\end{align*}
Therefore $T_j\star_{\bft^{(2)}}\mu\left(\Psi(\bft^{(2)})\right)\in H'_{\bft^{(2)}}$. This finishes the proof of the lemma.
\end{proof}

We will use the Laplace-dual connection $\widehat{\nabla}^{{\rm{sm}}}$ \cite[Section 2.5]{GGI16} of $\nabla^{\rm{sm}}$ to study $\nabla^{\rm{sm}}$-flat sections. Let $\bC_\lambda$ be a copy of $\bC$ with coordinate $\lambda$. The  connection $\widehat{\nabla}^{{\rm{sm}}}$ is a meromorphic flat connection on the trivial $H^*(X)$-bundle over $H^2(X)\times\bC_\lambda$. With respect to a frame of constant sections, it reads $$\widehat{\nabla}^{{\rm{sm}}}=\mathrm{d}-\suml_{i=1}^{s'}(T_i\star_{\bft^{(2)}})\left(\lambda-\hat{c}_1(\bft^{(2)})\right)^{-1}\mu\mathrm{d}t_i+\left(\lambda-\hat{c}_1(\bft^{(2)})\right)^{-1}\mu\mathrm{d}\lambda.$$ When $|\lambda| \gg 0$, we have $\left(\lambda-\hat{c}_1(\bft^{(2)})\right)^{-1}=\lambda^{-1}\suml_{k\ge0}\left(\lambda^{-1}\hat{c}_1(\bft^{(2)})\right)^k$. It follows that the connection $\widehat{\nabla}^{{\rm{sm}}}$ has logarithmic singularities along the smooth divisor $\{\lambda=\infty\}$.
    
\begin{lemma}\label{lemma--localholomorphicflatsectionofsmalldualconn}
   Suppose that  $(u,\Psi)$ is a simple pair over a domain $U\subset H^2(X)$, and let $P_0=(\bft^{(2)}_0,u(\bft^{(2)}_0))\in U\times\bC_\lambda$. Then near $P_0$ there is a unique holomorphic $\widehat{\nabla}^{{\rm{sm}}}$-flat section $\hat{y}(\bft^{(2)},\lambda)$ satisfying $\hat{y}(\bft^{(2)},u(\bft^{(2)}))=\Psi(\bft^{(2)})$.
\end{lemma}
\begin{proof}
    Consider the function $\bar{\lambda}(\bft^{(2)},\lambda):=\lambda-u(\bft^{(2)})$ on $U\times\bC_\lambda$. Then $\bar{\lambda},t_1,...,t_{s'}$ are holomorphic coordinates on $U\times\bC_{\lambda}$ and the divisor $\{\lambda=u\}$ is $\{\bar{\lambda}=0\}$. In terms of these new coordinates, the connection $\widehat{\nabla}^{{\rm{sm}}}$ with respect to a frame of constant sections reads
    \begin{align*}
    \widehat{\nabla}^{{\rm{sm}}}&=\mathrm{d}+\suml_{i=1}^{s'}\left(\partial_{t_i}u(\bft^{(2)})-(T_i\star_{\bft^{(2)}})\right)M(\bft^{(2)},\bar{\lambda})\mu\mathrm{d}t_i+M(\bft^{(2)},\bar{\lambda})\mu\mathrm{d}\bar{\lambda}.
    \end{align*}
   where $M(\bft^{(2)},\bar{\lambda}):=\left(\bar{\lambda}-(\hat{c}_1(\bft^{(2)})-u(\bft^{(2)}))\right)^{-1}$.

   We show that the endomorphisms
$M_i(\bft^{(2)},\bar{\lambda}):=\left(\partial_{t_i}u(\bft^{(2)})-(T_i\star_{\bft^{(2)}})\right)M(\bft^{(2)},\bar{\lambda})$  and $\bar{\lambda}M(\bft^{(2)},\bar{\lambda})$ are holomorphic near $P_0$, i.e., when $\bft^{(2)}$ is near $\bft^{(2)}_0$ and $\bar{\lambda}$ is near $0$. Around $P_0$, we let $\Psi(\bft^{(2)}),f_1(\bft^{(2)}),...,f_{s-1}(\bft^{(2)})$ be a local frame, and set $e_j:=f_j-(f_j,\Psi)\Psi$. Then $\underline{e}:=(\Psi,e_1,...,e_{s-1})$ is a local frame near $P_0$ satisfying $(\Psi,e_j)=0$, since $\left(\Psi,\Psi\right)=1$. It suffices to show that applying $M_i(\bft^{(2)},\bar{\lambda})$ and $\bar{\lambda}M(\bft^{(2)},\bar{\lambda})$ on $\underline{e}$, we get holomorphic sections. On the one hand,   $\Psi(\bft^{(2)})$ is an eigenvector of $\hat{c}_1(\bft^{(2)})$ with eigenvalue $u(\bft^{(2)})$, implying $M(\bft^{(2)},\bar{\lambda})\left(\Psi(\bft^{(2)})\right)=\bar{\lambda}^{-1}\Psi(\bft^{(2)})$. Consequently, from Lemma \ref{lemma-partialforuinsmallquantum}, we have 
    \[M_i(\bft^{(2)},\bar{\lambda})(\Psi(\bft^{(2)}))=0,\quad\bar{\lambda}M(\bft^{(2)},\bar{\lambda})(\Psi(\bft^{(2)}))=\Psi(\bft^{(2)}).
    \]
    On the other hand, let $H'_{\bft^{(2)}}\subset H^*(X)$ be the direct sum of generalized eigenspaces of $\hat{c}_1(\bft^{(2)})$ with eigenvalues $\neq u(\bft^{(2)})$, and set $M'(\bft^{(2)}):=\left(\hat{c}_1(\bft^{(2)})-u(\bft^{(2)})\right)|_{H'_{\bft^{(2)}}}$. Then $M'(\bft^{(2)})$ is invertible, and 
      \[
    M(\bft^{(2)},\bar{\lambda})|_{H'_{\bft^{(2)}}}=(-1)^{s-1}M'(\bft^{(2)})^{-1}\suml_{k\ge0}\left(\bar{\lambda}M'(\bft^{(2)})^{-1}\right)^k,
    \]
when $|\bar{\lambda}|$ is sufficiently small. Note that $(\Psi,e_j)=0$ gives $e_j(\bft^{(2)})\in H'_{\bft^{(2)}}$ by Lemma \ref{lemma-linearalgebra} (ii). As a consequence, the section  $M(\bft^{(2)},\bar{\lambda})e_j(\bft^{(2)})$ is holomorphic  near $P_0$. This proves that $M_i(\bft^{(2)},\bar{\lambda})$ and $\bar{\lambda}M(\bft^{(2)},\bar{\lambda})$ are holomorphic near $P_0$.

    So near $P_0$, the connection $\widehat{\nabla}^{{\rm{sm}}}$ is either holomorphic or has logarithmic singularities along $\{\bar{\lambda}=0\}$, and we only need to consider the singular case. The residue of $\widehat{\nabla}^{{\rm{sm}}}$ along $\{\bar{\lambda}=0\}$ is the following endomorphism of the trivial $H^*(X)$-bundle over $\{\bar{\lambda}=0\}$:
    \[
R(\bft^{(2)}):=\left(\bar{\lambda}M(\bft^{(2)},\bar{\lambda})\mu\right)_{\bar{\lambda}=0}.
    \]
    From the discussions in the preceding paragraph, we see that 
    \[
    \left(\bar{\lambda}M(\bft^{(2)},\bar{\lambda})\right)_{\bar{\lambda}=0}(\Psi(\bft^{(2)}))=\Psi(\bft^{(2)}),\quad\left(\bar{\lambda}M(\bft^{(2)},\bar{\lambda})\right)_{\bar{\lambda}=0}|_{H'_{\bft^{(2)}}}=0.
    \]
    Moreover, degree-counting gives $(\mu(\Psi(\bft^{(2)})),\Psi(\bft^{(2)}))=0$, implying $\mu(\Psi(\bft^{(2)}))\in H'_{\bft^{(2)}}$ since $H'_{\bft^{(2)}}$ is the orthogonal complement of $\bC\Psi(\bft^{(2)})$ with respect to the Poincar\'e pairing. Now we can check that $R(\bft^{(2)})\Psi(\bft^{(2)})=0$ and $R(\bft^{(2)})^2e_j(\bft^{(2)})=0$. So $R$ is nilpotent, and as a consequence, a section near $P_0$ is $\widehat{\nabla}^{{\rm{sm}}}$-flat if and only if it is of the form (see, e.g., \cite[Theorem 2 and Remark 2]{YT75}):
    \begin{align}\label{eq--fundamentalsolutionofLaplacedualofsmallconnection}
   (\bft^{(2)},\bar{\lambda})\mapsto \hat U(\bft^{(2)},\bar{\lambda})\exp\left(-R(\bft^{(2)}_0)\log\bar{\lambda}\right)(v),\quad v\in H^*(X).
    \end{align}
    Here $\hat U(\bft^{(2)},\bar{\lambda})$ is a holomorphic endomorphism near $P_0$ with $\hat U(\bft^{(2)}_0,0)=\mathrm{id}$, and $R(\bft^{(2)}_0)$ is viewed as a constant endomorphism near $P_0$. Let $\hat{y}(\bft^{(2)},\bar{\lambda})$ be the $\widehat{\nabla}^{{\rm{sm}}}$-flat near $P_0$ with $v=\Psi(\bft^{(2)}_0)$. Then it is holomorphic near $P_0$ since $R(\bft^{(2)}_0)\Psi(\bft^{(2)}_0)=0$.

    It remains to show that $\hat{y}(\bft^{(2)},0)=\Psi(\bft^{(2)})$. Let $\widehat{\nabla}^{\mathrm{sm,res}}$ be the residual connection of $\widehat{\nabla}^{{\rm{sm}}}$ along the divisor $\{\bar{\lambda}=0\}$, and with respect to a frame of constant sections it reads $\widehat{\nabla}^{\mathrm{sm,res}}=\mathrm{d}+\suml_{i=1}^{s'}M_i(\bft^{(2)},0)\mu\mathrm{d}t_i$.
    Note that $\hat{y}(\bft^{(2)},0)$ is a $\widehat{\nabla}^{\mathrm{sm,res}}$-flat section over $\{\bar{\lambda}=0\}$ and $\hat{y}(\bft^{(2)}_0,0)=\Psi(\bft^{(2)}_0)$. So it suffices to show that $\partial_{t_i}\Psi(\bft^{(2)})+M_i(\bft^{(2)},0)\mu\Psi(\bft^{(2)})=0$. Note that both terms are in $H'_{\bft^{(2)}}$, and we only need to prove that $M'(\bft^{(2)})\left(\partial_{t_i}\Psi(\bft^{(2)})+M_i(\bft^{(2)},0)\mu\Psi(\bft^{(2)})\right)=0$, or equivalently, \[\left(u(\bft^{(2)})-\hat{c}_1(\bft^{(2)})\right)\partial_{t_i}\Psi(\bft^{(2)})+\left(\partial_{t_i}u(\bft^{(2)})-(T_i\star_{\bft^{(2)}})\right)\mu\Psi(\bft^{(2)})=0.\] This follows from the flatness of $\widehat{\nabla}^{{\rm{sm}}}$ and Lemma \ref{lemma-partialforuinsmallquantum}.
    
    This proves the existence. Uniqueness follows from \eqref{eq--fundamentalsolutionofLaplacedualofsmallconnection}. 
\end{proof}

We  use Definition \ref{def-respect} for   elements in $\mathcal{S}^{\rm sm}$ with exactly the same words.
\begin{lemma}\label{lemma: existy}
    Suppose that  $(u,\Psi)$ is a simple rightmost pair over a domain $U$ inside the $(\mathrm{SR})$-region of $X$. Then there exists a unique $y\in \mathcal{S}^{\rm sm}$ such that for any $\bft^{(2)}_0\in U$, $y$ respects $(u,\Psi)$ over an open neighborhood of $\bft^{(2)}_0$ with phase zero.
 
\end{lemma}

\begin{proof}
     Let $\hat{y}_{\bft^{(2)}_0}(\bft^{(2)},\lambda)$ be the $\widehat{\nabla}^{{\rm{sm}}}$-flat section from Lemma \ref{lemma--localholomorphicflatsectionofsmalldualconn}, which is holomorphic near $(\bft^{(2)}_0,u(\bft^{(2)}_0))$ and satisfies $\hat{y}(\bft^{(2)},u(\bft^{(2)}))=\Psi(\bft^{(2)})$. Consider the Laplace transform
    \[
    \bar{y}_{\bft^{(2)}_0}(\bft^{(2)},z):={1\over z}\int_{u(\bft^{(2)})+\bR_{\ge0}\rme^{\mathbf{i}0}}\hat{y}_{\bft^{(2)}_0}(\bft^{(2)},\lambda)\rme^{-{\lambda\over z}}\mathrm{d}\lambda, \quad |\arg z|<{\pi\over 2}, \bft^{(2)}\text{ near }\bft^{(2)}_0.
    \]
Note that $\hat{y}_{\bft^{(2)}_0}(\bft^{(2)},\lambda)$ is holomorphic near $\lambda=u(\bft^{(2)})$, and has moderate growth as $\lambda\to\infty$ since $\widehat{\nabla}^{{\rm{sm}}}$ has logarithmic singularities along the smooth divisor $\{\lambda=\infty\}$. Therefore the Laplace transform is well defined. Slightly changing the slope of the integration path, we can analytically continue $\bar{y}_{\bft^{(2)}_0}(\bft^{(2)},z)$ to $|\arg z|<\frac{\pi}{2}+\varepsilon$ for sufficiently small $\varepsilon>0$. Using integration by parts, we see that $\bar{y}_{\bft^{(2)}_0}(\bft^{(2)},z)$ is $\widetilde{\nabla}^{\rm{sm}}$-flat and respects $(u,\Psi)$ with phase zero over an open neighborhood of $\bft^{(2)}_0$. When $\bft^{(2)}_1$ is sufficiently close to $\bft^{(2)}_0$, it follows from  $\hat{y}_{\bft^{(2)}_0}(\bft^{(2)},\lambda)=\hat{y}_{\bft^{(2)}_1}(\bft^{(2)},\lambda)$ that $\bar{y}_{\bft^{(2)}_0}(\bft^{(2)},z)=\bar{y}_{\bft^{(2)}_1}(\bft^{(2)},z)$. This proves the existence part of the lemma.

    For the uniqueness, suppose that $y_1(\bft^{(2)},z)$ and $y_2(\bft^{(2)},z)$ are two such $\widetilde{\nabla}^{\rm{sm}}$-flat sections, and $\bft^{(2)}_0\in U$. Then we have $\rme^{{u(\bft^{(2)}_0)\over z}}y_1(\bft^{(2)}_0,z),\rme^{{u(\bft^{(2)}_0)\over z}}y_2(\bft^{(2)}_0,z)\to\Psi(\bft^{(2)}_0)$, as $z\to0$ along  the positive real axis. Then the functions $y_1(\bft^{(2)}_0,z)$ and $y_2(\bft^{(2)}_0,z)$ of $z$ are in $\mathcal{A}(\bft^{(2)}_0)$ and equal to each other, since $\dim_\bC\mathcal{A}(\bft^{(2)}_0)=1$. As a consequence, we have $y_1(\bft^{(2)},z)=y_2(\bft^{(2)},z)$. This proves the uniqueness part of the lemma.
\end{proof}

\begin{theo}\label{thm--gammaIoverdomain}
    Suppose that $U$ is a domain inside the $(\mathrm{SR})$-region of $X$. If $X$ satisfies Gamma-I at some $\bft^{(2)}_0\in U$, then $X$ satisfies Gamma-I over $U$.
\end{theo}
\begin{proof}
 For any $\bft^{(2)}_1\in U$, let $U'\subset U$ be a simply connected domain containing both $\bft^{(2)}_0$ and $\bft^{(2)}_1$, and such that $(u,\Psi)$ is a simple rightmost pair over $U'$. From Lemma \ref{lemma: existy}, let $y$ be the $\widetilde{\nabla}^{\rm{sm}}$-flat section respecting $(u,\Psi)$ over $U'$ with phase zero. Then $y\in\mathcal{A}(\bft^{(2)}_0)$. Since $\mathcal{Z}^{K,{\rm{sm}}}(\cO)\in\mathcal{A}(\bft^{(2)}_0)$ and $\dim_\bC\mathcal{A}(\bft^{(2)}_0)=1$, it follows that there exists a nonzero $c\in\bC$ satisfying $\mathcal{Z}^{K,{\rm{sm}}}(\cO)=cy$. Now $y\in\mathcal{A}(\bft^{(2)}_1)$ implies that $\mathcal{Z}^{K,{\rm{sm}}}(\cO)\in\mathcal{A}(\bft^{(2)}_1)$. This finishes the proof of the theorem.
\end{proof}

\subsection{AEFS for $\widetilde{\nabla}^{\rm{sm}}$ and the strategy-type theorem}\label{sec-AEFS and strategy}

Let $K_{ts}\subset H^2(X)$ be the subset of  {tame semisimple} points in the K\"ahler moduli, i.e. for $\bft^{(2)}\in K_{ts}$, $\hat{c}_1(\bft^{(2)})$ has only simple eigenvalues. In this subsection, we assume that $K_{ts}\ne\varnothing$, in which case $H^2(X)\setminus K_{ts}$ is a divisor of $H^2(X)$.

For a sufficiently small domain $U\subset K_{ts}$, by the Implicit Function Theorem, there are holomorphic maps $U\xrightarrow{u_i}\bC_\lambda$ and $U\xrightarrow{\psi_i} H^*(X)$ ($1\le i\le s$), such that for each $\bft^{(2)}\in U$, $u_i(\bft^{(2)})$ is an eigenvalue of $\hat{c}_1(\bft^{(2)})$ and $\psi_i(\bft^{(2)})$ is the idempotent generator of the associated one-dimensional eigenspace. The functions $u_i$'s, and hence $\psi_i$'s, are unique up to ordering. When the normalized idempotent $\Psi_i:=(\psi_i,\psi_i)^{-{1\over 2}}\psi_i$ is  a well-defined holomorphic map  (which holds, for instance, if $U$ is simply connected), we say that $U$ is properly-chosen with respect to $\{(u_i,\Psi_i)\}$.

\begin{lemma}\label{lemma-holomorphicflatsectionwrtlaplacedualconn}
    Suppose that $U\subset K_{ts}$ is properly-chosen with respect to $\{(u_i,\Psi_i)\}$, and $\bft^{(2)}_0\in U$. 
 Then near $P_i:=(\bft^{(2)}_0,u_i(\bft^{(2)}_0))\in U\times\bC_\lambda$, there is a unique holomorphic $\widehat{\nabla}^{\rm{sm}}$-flat section $\hat{y}_i(\bft^{(2)},\lambda)$ satisfying $\hat{y}_i(\bft^{(2)},u_i(\bft^{(2)}))=\Psi_i(\bft^{(2)})$.
\end{lemma}
\begin{proof}
    This follows from Lemma \ref{lemma--localholomorphicflatsectionofsmalldualconn} since each $(u_i,\Psi_i)$ is a simple pair over $U$.
\end{proof}
\begin{remark}\label{rmk39}
    Here we do not assume $\bft^{(2)}_0\in B$, and so the proof does not use the assumption that $\{u_i\}$ is the canonical coordinate near $\bft^{(2)}_0$.
\end{remark}

Let $\phi\in\bR$ be an admissible phase at $\bft^{(2)}_0$. Now we follow \cite[Section 2]{GGI16} to use the Laplace transform of the above $\{\hat{y}_i\}$ to construct the AEFS $\{y_i^\phi\}$ of $\nabla^{\rm{sm}}$ near $\bft^{(2)}_0$ associated to the phase $\phi$ with respect to $\{\Psi_i\}$. Let $L_i(\bft^{(2)},\phi):=u_i(\bft^{(2)})+\bR_{\ge0}\rme^{\mathbf{i}\phi}$. Set
\[y_i^\phi(\bft^{(2)},z):={1\over z}\int_{L_i(\bft^{(2)},\phi)}\hat{y}_i(\bft^{(2)},\lambda)\rme^{-{\lambda\over z}}\mathrm{d}\lambda,\quad |\arg z-\phi|<{\pi\over 2},\bft^{(2)}\text{ near }\bft^{(2)}_0.\]
The AEFS $y_1^\phi,\dots,y_{s}^\phi$ near $\bft^{(2)}_0$  are semiorthonormal with respect to the phase $\phi$ in the sense that $[y_i^\phi,y_i^\phi)=1$ and $[y_i^\phi,y_j^\phi)=0$ if $\Im(\rme^{-\mathbf{i}\phi}u_i(\bft^{(2)}_0))<\Im(\rme^{-\mathbf{i}\phi}u_j(\bft^{(2)}_0))$.

The above $L_i(\bft^{(2)}_0,\phi)$ is an admissible path at $\bft^{(2)}_0$ starting from $u_i(\bft^{(2)}_0)$. For a regular embedding $$\gamma: \bR_{\ge0}\rightarrow \bC_\lambda,$$i.e. a smooth embedding with nowhere-vanishing $\gamma'(t)$, we say that it is \textbf{admissible at $\bft^{(2)}_0$ with end-phase $\phi\in\bR$}, if $\gamma(0)\in\Spec(\hat{c}_1(\bft^{(2)}_0))$, $\gamma(\bR_{>0})\cap\Spec(\hat{c}_1(\bft^{(2)}_0))=\varnothing$ and $\gamma(t)$ satisfies $\arg\gamma'(t)=\phi$ when $t\gg0$. When $\gamma(0)=u_i(\bft^{(2)}_0)$, the following Laplace transform $y^\gamma(z):={1\over z}\int_\gamma\hat{y}_i(\bft^{(2)}_0,\lambda)\rme^{-{\lambda\over z}}\mathrm{d}\lambda$ when $|\arg z-\phi|<{\pi\over 2}$, is in $\mathcal{S}^{\rm{sm}}_{\bft^{(2)}_0}:=\{y(\bft^{(2)}_0,z)\mid\widetilde{\nabla}^{\rm{sm}}y(\bft^{(2)},z)=0\}$, and there is a unique $y^{\gamma}(\bft^{(2)},z)\in\mathcal{S}^{\rm{sm}}$ such that $y^{\gamma}(\bft^{(2)}_0,z)=y^\gamma(z)$. Note that $\hat{y}_i(\bft^{(2)},\lambda)$ is the unique $\widehat{\nabla}^{\rm{sm}}$-flat section holomorphic near $(\bft^{(2)}_0,u_i(\bft^{(2)}_0))$ satisfying $\hat{y}_i(\bft^{(2)}_0,u_i(\bft^{(2)}_0))=\Psi_i(\bft_0^{(2)})$ by Lemma \ref{lemma-holomorphicflatsectionwrtlaplacedualconn}. We say that $y^\gamma\in\mathcal{S}^{\rm{sm}}$ is \textbf{represented by $\gamma$ in $\mathbb{C}_\lambda$ at $\bft^{(2)}_0$ with respect to $(u_i(\bft^{(2)}_0),\Psi_i(\bft^{(2)}_0))$}. Note that two $\widetilde{\nabla}^{\rm{sm}}$-flat sections represented by the same admissible path at $\bft^{(2)}_0$  differ at most by a sign, due to sign ambiguity of normalized idempotents.

Mutations of $L_i(\bft^{(2)}_0,\phi)$ give mutations of $y_i^\phi$. The right mutation of $L_i(\bft^{(2)}_0,\phi)$ with respect to $u_j(\bft^{(2)}_0)$ is depicted in Figure \hyperref[figure-rightmuation]{1(A)}. During the deformation  from  $L_i(\bft^{(2)}_0,\phi)$ to $L'_i$, the path is assumed to cross only one eigenvalue $u_j(\bft^{(2)}_0)$. Let $y'_i$ be the corresponding right mutation of $y_i^\phi$, which is determined by the Laplace transform of $\hat{y}_i(\bft^{(2)}_0,\lambda)$ over $L'_i$. Then $y'_i=y^\phi_i-[y^\phi_i,y^\phi_j)\cdot y^\phi_j$. Similarly, let $y''_i$ be the left mutation of $y_i$ with respect to $u_k(\bft^{(2)}_0)$ as depicted in Figure \hyperref[figure-leftmuation]{1(B)}, and we have $y''_i=y^\phi_i-[y^\phi_k,y^\phi_i)\cdot y^\phi_k$. We refer the reader to \cite[Section 2]{GGI16} for details.

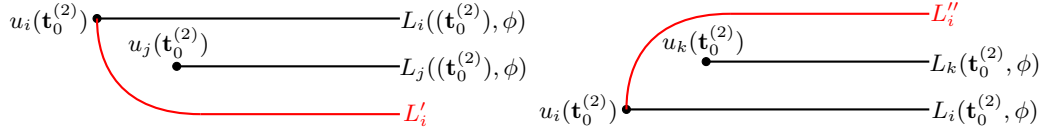
\begin{figure}[htbp]

    \centering
\begin{minipage}{0.55\textwidth}
\begin{flushleft}
  \begin{tikzpicture}[x=30pt,y=30pt]
\draw (0,0) node {\phantom{A}};

\filldraw (1.2, 0) circle [radius =0.05];
\draw (1.1,0.3) node {\small $u_j(\bft^{(2)}_0)$};

\filldraw (0.2, 0.6) circle [radius =0.05];
\draw (-0.4,0.6) node {\small $u_i(\bft^{(2)}_0)$};

\draw[thick] (1.2,0) -- (4,0);

\draw[thick] (0.2,0.6) -- (4,0.6);

\draw[color=red,thick] (0.2,0.6) .. controls  (0.2,0) and (0.5,-0.6) .. (1.5,-0.6);

\draw[color=red,thick] (1.5,-0.6) -- (4,-0.6);

\draw (4.8,0) node {\small$L_j((\bft^{(2)}_0),\phi)$};
\draw (4.8,0.6) node {\small$L_i((\bft^{(2)}_0),\phi)$};

\draw[color=red] (4.2,-0.6) node {\small$L_i'$};

\end{tikzpicture}
\subcaption{\footnotesize Right mutation of $L_i((\bft^{(2)}_0),\phi)$ with respect to $u_j(\bft^{(2)}_0)$.} 
\label{figure-rightmuation}

\end{flushleft}
\end{minipage}%
\begin{minipage}{0.55\textwidth}
  \centering
  
\begin{tikzpicture}[x=30pt,y=30pt]
\draw (0,0) node {\phantom{A}};

\filldraw (1.2, 0) circle [radius =0.05];
\draw (1.17,0.27) node {\small $u_k(\bft^{(2)}_0)$};

\filldraw (0.2, -0.6) circle [radius =0.05];
\draw (-0.4,-0.6) node {\small $u_i(\bft^{(2)}_0)$};

\draw[thick] (1.2,0) -- (4,0);

\draw[thick] (0.2,-0.6) -- (4,-0.6);

\draw[color=red,thick] (0.2,-0.6) .. controls  (0.2,0) and (0.5,0.6) .. (1.5,0.6);

\draw[color=red,thick] (1.5,0.6) -- (4,0.6);

\draw (4.7,0) node {\small$L_k(\bft^{(2)}_0,\phi)$};
\draw (4.7,-0.6) node {\small$L_i(\bft^{(2)}_0,\phi)$};

\draw[color=red] (4.2,0.6) node {\small$L_i''$};

\end{tikzpicture}
\subcaption{\footnotesize Left mutation of $L_i(\bft^{(2)}_0,\phi)$ with respect to $u_k(\bft^{(2)}_0)$.}\label{figure-leftmuation}
\end{minipage}
\vspace{-0.5em}
\caption{Mutations.} \label{fig:1}
\end{figure}
For completeness, we restate the following and includes its proof.
\begin{thm}[Strategy-Theorem]\label{thm:strategy}
    For  a  Fano manifold $X$, assume  the followings. 
    \begin{enumerate}[leftmargin=0pt, itemindent=0.5em, labelsep=0.5em]
    \item The big quantum cohomology of $X$ is convergent near the large radius limit.
    \item There exists $\bft^{(2)}_0\in H^2(X)$ at which $X$ satisfies property Gamma-I.
       \item There exists a domain $U$ inside the $(\mathrm{SR})$-region of $X$ such that  $\bft^{(2)}_0\in U$ and $U\cap K_{ts}\neq \varnothing$.
        
    \item  Denote $V_1:=\cO$. There exists $\bft^{(2)}_1\in U\cap K_{ts}$ such that:
    \begin{enumerate}
    \item there exist   $V_2,\dots,V_{s} \in \mathcal{D}^b_{\rm coh}(X)$ such that each $V_i$ is obtained from $\cO$ via the Galois action along some path $\gamma_i\subset H^2(X)$ starting at the same point  $\bft^{(2)}_1$;
        \item there exists a full exceptional collection $(\tilde{V}_1, \dots, \tilde{V}_s)$ in $\mathcal{D}^b_{\rm coh}(X)$, obtained from    the set $\{V_1, \dots, V_s\}$ by ``elementary operations";
        \item  there exists an admissible phase $\phi_1$ at $\bft^{(2)}_1$ such that $\mathcal{Z}^{K, \mathrm{sm}}(\tilde{V}_i)$ respects the pair $(u_i, \Psi_i)$ at $\bft^{(2)}_1$ with phase $\phi_1$ for all $i$, where the $u_i$ are ordered so that $\Im(\rme^{-\mathbf{i}\phi}u_1(\bft^{(2)}_1))>\dots>\Im(\rme^{-\mathbf{i}\phi}u_{s}(\bft^{(2)}_1))$.
     \end{enumerate}
\end{enumerate}    
    Then Gamma conjecture II holds for $X$. 
\end{thm}

\begin{proof}
   Let $B$ be the  domain of convergence of the big quantum cohomology of $X$ and $B_{ss}$ the semisimple locus, as in Section \ref{sec--preliminary}. Observe that $B\cap K_{ts}\subset B_{ss}$. Note that $B\cap H^2(X)\subset H^2(X)$ is a domain, and $H^2(X)\setminus K_{ts}\subset H^2(X)$ is a divisor. Therefore $B\cap K_{ts}\ne\varnothing$, implying $B_{ss}\ne\varnothing$. Let $\bft^{(2)}_2\in B\cap K_{ts}\subset B_{ss}$, and let $\phi_2\in\bR$ be an  admissible phase at $\bft^{(2)}_2$.  \par
   We claim that, as $(\bft^{(2)},\phi)$ varies from $(\bft^{(2)}_1,\phi_1)$ to $(\bft^{(2)}_2,\phi_2)$ inside $K_{ts}$, the corresponding AEFS of $\nabla^{\rm sm}$ changes by mutations, compatibly with the braid group action on full exceptional collections. To see this, let $\bft^{(2)}\in K_{ts}$ be on the path connecting $\bft^{(2)}_1$ and $\bft^{(2)}_2$, and let a neighborhood of $\bft^{(2)}$ be  properly-chosen with respect to $\{(\bar{u}_i,\bar{\Psi}_i)\}$. Let $\phi,\phi'\in\bR$ be admissible phases at $\bft^{(2)}$, and suppose that $(E_1,\dots,E_s)$ is full exceptional collections in $\mathcal{D}^b_{\rm coh}(X)$ that give AEFS of $\nabla^{\rm{sm}}$ at $\bft^{(2)}$ with respect to $\{\bar{\Psi}_i\}$ associated to phases $\phi$.  Suppose $\phi'>\phi$, and without loss of generality assume that they are close enough so that when we rotate $L_i(\bft^{(2)},\phi)$ counterclockwise around its endpoint $\bar{u}_i(\bft^{(2)})$ by the angle $\phi'-\phi$ to get $L_i(\bft^{(2)},\phi')$ for $1\le i\le s$, only the sector spanned by deformation of $L_{i_0}(\bft^{(2)},\phi)$ contains an eigenvalue of $\hat{c}_1(\bft^{(2)})$; in this case, the eigenvalue is $\bar{u}_{i_0-1}(\bft^{(2)})$. Now according to the discussion preceding Figure \ref{fig:1}, the AEFS of  $\nabla^{\rm{sm}}$ at $\bft^{(2)}$ with respect to $\{\bar{\Psi}_i\}$ associated to phases $\phi'$ is
   \begin{align*}
   &\mathcal{Z}^{K,\rm{sm}}(E_1),\dots,\mathcal{Z}^{K,\rm{sm}}(E_{i_0-2}),\\
   &\mathcal{Z}^{K,\rm{sm}}(E_{i_0})-[\mathcal{Z}^{K,\rm{sm}}(E_{i_0-1}),\mathcal{Z}^{K,\rm{sm}}(E_{i_0}))\mathcal{Z}^{K,\rm{sm}}(E_{i_0-1}),\mathcal{Z}^{K,\rm{sm}}(E_{i_0-1}),\\
   &\mathcal{Z}^{K,\rm{sm}}(E_{i_0+1}),\dots,\mathcal{Z}^{K,\rm{sm}}(E_s).
   \end{align*}
This is the AEFS given by the full exceptional collection obtained by applying $\sigma_{i_0-1}^{-1}$ on $(E_1,\dots,E_s)$. Conversely, if $\phi'<\phi$  are close enough so that only the sector corresponding to $L_{i_0}(\bft^{(2)},\phi)$ contains an eigenvalue of $\hat c_1(\bft^{(2)})$, then the rotation is clockwise. Similarly to the previous case, the AEFS mutation is given by applying $\sigma_{i_0+1}$ to the same exceptional collection. This verifies the claim (see \cite[Section 4]{GGI16} for systematic discussions). The proof is completed by observing that AEFS of $\nabla^{\rm{sm}}$ coincides with AEFS of $\nabla$ near $\bft^{(2)}_2\in K_{ts}\cap B$.
\end{proof}

\section{Gamma conjecture II holds for  del Pezzo surfaces}

In this section, we 
investigate   del Pezzo surfaces $X_r$ ($1\le r\le 8$), namely the blow-up of $\mathbb{P}^2$ at $r$ general points.
 It is known that  the big quantum cohomology   converges 
 near the large radius limit by \cite[Corollary 5.9]{Iri07}, and   is generically semisimple by \cite[Theorem 3.6.1]{BM04}. Therefore,   Gamma conjecture II is well defined for $X_r$. 
 The main result of this section is the following theorem, which we obtain  by verifying    all   assumptions of the Strategy-Theorem in Section 4.2. 
 \begin{thm}\label{thm-GammaIIforXr}
    Gamma conjecture II holds for $X_r$ for $1\le r\le 8$.
\end{thm}

 We start with some concrete information on the small quantum  cohomology of $X_r$. Let $H$ denote the pullback to $X_r$ of the hyperplane class on $\mathbb{P}^2$, and let $E_1,...,E_r$ be the exceptional divisors. Let $\mathbf{1}\in H^0(X_r)$ be the Poincar\'e dual of the fundamental class, and $[pt]\in H^4(X_r)$ the Poincar\'e dual of a point. Let $$T_0=\mathbf{1},T_1=H-E_1,\dots,T_r=H-E_r,T_{r+1}=H,T_{r+2}=[pt].$$ Then these classes form a basis of $H^*(X_r)$, and $T_1,\dots,T_{r+1}$ form a nef basis of $H^2(X_r,\bZ)$. We write $\bft=\suml_{i=0}^{r+2}t_iT_i$, and set $q_i:=\rme^{t_i}(1\le i\le r), q:=\rme^{t_{r+1}}$. 

 The dual basis of $\{T_1,\dots,T_{r+1}\}$ in $H_2(X_r,\bZ)$ is given by $\{E_1,\dots,E_r,H-\sum_iE_i\}$. So we can write $\bfd\in\Eff(X_r)$   as  $\bfd=\sum_ik_iE_i+k(H-\sum_iE_i)$, and  denote   
\begin{align}\label{notation-q^d}
\rme^{\bft^{(2)}(\bfd)}=q_1^{k_1}\cdots q_r^{k_r}q^k=:\bfq^{\bfd}, \mbox{  where }k_i,k\in\bZ_{\ge0}.    
\end{align}

For $\bfd\in\Eff(X_r)$ and $k\in\bZ_{\ge0}$, the invariant $\langle([pt])^k\rangle_\bfd^{X_r}$ is a nonnegative integer {\cite[Section 4.1]{GP98}}. Moreover, we will use the following Gromov-Witten invariants (see \cite[Lemma 3.3, Theorem 4.1]{GP98}, \cite[Proposition 3.4.1]{BM04}): 
\begin{align}\label{eq--invariants}
\langle \rangle^{X_r}_{E_i}=\langle \rangle^{X_r}_{H-E_i-E_j}=\langle[pt] \rangle^{X_r}_{H-E_i}=\langle([pt])^2\rangle^{X_r}_H=1, \quad i\neq j.
\end{align}

\subsection{Distribution of the eigenvalues of $\hat{c}_1(\bft^{(2)})$}
In this subsection, we study the distribution of the eigenvalues of $\hat{c}_1(\bft^{(2)})$. We first assume the following proposition, whose proof is the most technically involved  part and is deferred to Section \ref{section-condition*forXr}. It is needed both in the formulation and  proof of Proposition \ref{prop-estimateofeigevalueswhenrotating}, as well as in the proof of  Theorem \ref{thm-GammaIIforXr}. 

\begin{prop}\label{prop-main-simple rightmost} There exists $\delta_r\in(0,{1\over2}]$, such that the  set $$V_{\delta_r}:=\{\bft^{(2)}\in H^2(X_r, \mathbb{R})|0<q(\bft^{(2)})\le1\text{ and }1\le q_i(\bft^{(2)})\le 1+\delta_r, \, i=1,\dots,r\}$$
is contained in the $(\mathrm{SR})$-region of $X_r$. 
\end{prop}

Let $M_r$ be the matrix of the operator $\hat{c}_1(\bft^{(2)})$ on the small quantum cohomology of $X_r$ with respect to the ordered basis $[\mathbf{1},H,E_1,...,E_r,[pt]]$. Then the entries of $M_r$  lie in $\bR[q,q_1,\dots,q_r]$. We study the asymptotic behavior of the spectrum of $\hat{c}_1(\bft^{(2)})$, namely  the roots of  the characteristic polynomial   $\chi_{M_r}(u):=\det(u\cdot I_{r+3}-M_r)$.
  
\begin{exam}\label{example-r=1}
  From the data in \cite[Section 2.2.3]{HKLY21}, the matrices $M_1$ and $M_2$ are respectively given by 
    \begin{equation*}
 \begin{bmatrix}
            0& 2q& 2q&3qq_1\\
            3& 0& 0 &2q\\
            -1& 0& -q_1 & -2q\\
            0& 3& 1& 0
            
        \end{bmatrix}, \,\,\,\,  
 \begin{bmatrix}
            0& 2q(q_1+q_2)& 2qq_2& 2qq_1&3qq_1q_2\\
            3& q& q& q &2q(q_1+q_2)\\
            -1& -q& -q_1-q& -q & -2qq_2\\
            -1& -q& -q& -q_2-q&-2qq_1\\
            0& 3& 1& 1&0
        \end{bmatrix}.
    \end{equation*} 
 
\end{exam}

\begin{lemma}\label{proposition-matrix about q^0 and q^1}
    Let $1\leq r\leq 8$. Then modulo $q^2$, the matrix $M_r$ has the form
    \begin{equation*}\label{equ-matrix form of M_r}
        \begin{bmatrix}
         0& 2q(\sum\limits_{k=1}^r\prod\limits_{i\neq k} q_i) &2q\prod\limits_{i\neq 1} q_i & \cdots & 2q\prod\limits_{i\neq r} q_i  &  3q\prod\limits_{i=1}^r q_i\\
         3 &   q(\sum\limits_{k< l}\prod\limits_{i\neq k, l} q_i)  & q(\sum\limits_{k\neq 1}\prod\limits_{i \neq 1,k}q_i) &\cdots & q(\sum\limits_{k\neq 1}\prod\limits_{i \neq r,k}q_i) &   2q(\sum\limits_{k}\prod _{i\neq k} q_i) \\
         -1& -q(\sum\limits_{k\neq 1}\prod\limits_{i \neq 1,k}q_i)  & -q_1-q(\sum\limits_{k\neq 1}\prod\limits_{i \neq 1,k}q_i)  & \cdots & -q(\prod\limits_{i\neq 1,r}q_i)& -2q(\sum\limits_{k\neq 1}\prod\limits_{i\neq k}q_i) \\
         \cdots & \cdots  & \cdots  & \cdots & \cdots &  \cdots \\
         -1& -q(\sum\limits_{k\neq 1}\prod\limits_{i \neq 1,k}q_i) & -q(\prod\limits_{i\neq r,1}q_i) &\cdots & -q_r-q(\sum\limits_{k\neq r}\prod\limits_{i\neq r,k} q_i) & -2q(\sum\limits_{k\neq r}\prod\limits_{i\neq k}q_i)\\
         0 & 3 & 1 & \cdots &1   & 0
        \end{bmatrix},
    \end{equation*}
    where, for $1\leq j,j_1, j_2\leq r$ with $j_1\neq j_2$, we have

    \[(M_r)_{2+j,2+j}=-q_{j}-q(\sum_{k\neq j}\prod_{i\neq j,k}q_i)+O(q^2),\quad (M_r)_{2+j_1,2+j_2}=-q(\prod_{i\neq j_1,j_2}q_i)+O(q^2).\]  
   Furthermore,  $\det M_r=(-1)^r27\left(\prodl_{i=1}^rq_i\right)^2q+O(q^2)$.
\end{lemma}
\begin{proof}
    The matrix form follows from \eqref{eq--invariants} and the divisor axiom of Gromov-Witten theory. Let $s=r+3$. We have 
    \[
    \det M_r=\suml_{(i_1\dots i_s)}\mathrm{sign}(i_1\dots i_s)\cdot(M_r)_{i_1,1}\cdots(M_r)_{i_s,s}.
    \]
    Note that $q|(M_r)_{i,r+3}$ for all $i$. So for any term labeled by $(i_1\dots i_s)$ contributing to $\det M_r$ modulo $q^2$, we must have $i_2=r+3$, which forces $i_j=j$ for $3\leq j\leq r+2$. Then $i_1=2$ since $q|(M_r)_{11}$, implying that $i_{r+3}=1$. Consequently,
    \begin{align*}
    \det M_r&\equiv(M_r)_{21}(M_r)_{r+3,2}(M_r)_{33}\cdots(M_r)_{r+2,r+2}(M_r)_{1,r+3}\mod q^2\\
    &\equiv 3\cdot 3\cdot(-q_1)\cdots(-q_r)\cdot3q\prodl_{i=1}^rq_i\mod q^2.
    \end{align*}
This completes the proof.   
\end{proof}

To analyze the spectrum of $\hat{c}_1(\bft^{(2)})$ as $q \to 0$, we examine the roots of the characteristic polynomial $\chi_{M_r}(u)$. Since the coefficients of this polynomial depend polynomially on $q$, standard results on algebraic functions imply that its roots admit Puiseux expansions. These are   generalized power series that allow fractional exponents in $q$.\par

The following Propositions \ref{theorem-puiseux series} and \ref{prop-estimateofeigevalueswhenrotating}  determine the leading terms of these root expansions. Finding this lowest-order behavior is a crucial first step: it gives us the analytic control needed to rigorously track the monodromy of the eigenvalues, which is the  condition (3) for applying the Strategy-Theorem, as we move along paths $\gamma_i\subset H^2(X_r)$. Denote 
\begin{equation}
      \omega:=\exp\left({2\pi\mathbf{i}\over 3}\right).  
\end{equation}

\begin{prop}\label{theorem-puiseux series} For fixed $q_1,...,q_r\in\bC^\times$, the Puiseux expansions in $q$ of the eigenvalues  of $M_r$ have the following asymptotic forms as $q\to0$:
\begin{align*}%\label{equ-Puiseux series}
u_k(q;q_1,\dots,q_r)&=3\omega^{k-1}(\prod q_i)^{\frac{1}{3}}q^{\frac{1}{3}}+o(q^{\frac{1}{3}}),&k=1,2,3,\\
        u_{3+k}(q;q_1,\dots,q_r)&=-q_k+o(1),& k=1,\dots,r.
\end{align*}
\end{prop}
\begin{proof} 
    By Lemma \ref{proposition-matrix about q^0 and q^1}, we see that $\chi_{M_r}(u)|_{q=0}=u^3\prod(u+q_i)$.
    Hence the Puiseux expansions in $q$ of the eigenvalues of $M_r$ have the following asymptotic forms as $q\to0$:
\begin{align*}
    u_k(q)&=C_kq^{w_k}+o(q^{w_k}),\quad k=1,2,3,(w_1\ge w_2\ge w_3>0,C_k\text{ is independent of }q)\\
        u_{3+k}(q)&=-q_k+o(1),\quad k=1,\dots,r.
\end{align*}
Therefore, in the Puiseux expansion of the product $\prod_{k=1}^{3+r}u_k(q)$, the lowest-order exponent of $q$ is $w_1+w_2+w_3$. Since $\prod_{k=1}^{3+r}u_k(q)=\det M_r$, Lemma \ref{proposition-matrix about q^0 and q^1} implies that
  \begin{align}\label{eq-w1w2w3}
   w_1+w_2+w_3=1.
  \end{align}
Assume $w_1>w_2$. As a Puiseux expansion in $q$, the lowest-order exponent of $\sum_{i=1}^{r+3}\prod_{k\neq i}u_k(q)$ is $w_2+w_3$. Note that $\sum_{i=1}^{r+3}\prod_{k\neq i}u_k(q)$ can be given by minors of $M_r$, and hence is a polynomial in $q$. So $w_2+w_3\geq 1$, contradicting  \eqref{eq-w1w2w3}. Thus $w_1=w_2$. Similarly, by considering $\sum_{i<j}\prod_{k\ne i,j}u_k(q)$, we see $w_1=w_2=w_3=\frac{1}{3}$.

       Consider the Puiseux expansions of the three terms $\sum_{i<j}\prod_{k\ne i,j}u_k(q)$, $\sum_{i=1}^{r+3}\prod_{k\neq i}u_i(q)$, $\prod_{k=1}^{3+r}u_k(q)$ in $q$. Note that these terms can be given by minors of $M_r$, and hence are polynomials in $q$. The powers of their lowest degree terms in $q$ then respectively give
         \begin{equation*}    
            C_1+C_2+C_3=0, \quad C_1 C_2+C_2 C_3+C_1C_3=0, \quad  C_1C_2C_3=27\prod q_i.
     \end{equation*}
 
    Hence $C_1, C_2, C_3$ are the three distinct roots of the cubic polynomial $C^3=27\prod q_i$. This proves the proposition.
\end{proof}

\begin{lemma}\label{prop- equality prop}
    There exists $\delta_{r*}>0$, such that whenever $|q|\in (0,\delta_{r*})$, $|q_i|\in[1,1+\delta_r]$, $\theta\in\bR$ and $1\leq k\leq 3$, the complex number $u^*_{k}:=3\omega^{k-1}q^{\frac{1}{3}}(\prod q_i)^{\frac{1}{3}}+{1\over 100}e^{\mathbf{i}\theta}q^{\frac{1}{3}}$ is not an eigenvalue of $M_r$.
\end{lemma}
\begin{proof}
    Recall that the coefficients of  $\chi_{M_r}(u)$ of $M_r$ are in $\bR[q,q_1,\cdots,q_r]$, and $\chi_{M_r}(u)|_{q=0}=u^3\prod(u+q_i)$. So we can write $\chi_{M_r}(u)$ in the following form:
    \begin{align*}
        u^{r+3}+\cdots+( \prod q_i+q\cdot a_{3}) u^3+q\cdot a_2u^2+q\cdot a_1u+ (-1)\cdot 27 \cdot(\prod q_i)^2q+q^2a_0,
    \end{align*}
where $a_i\in\bR[q,q_1,\dots,q_r]$.    Note that $u_k^*=q^{\frac{1}{3}}(3\omega^{k-1}(\prod q_i)^{\frac{1}{3}}+{1\over 100}e^{\mathbf{i}\theta})$, and we see that $\chi_{M_r}(u^*_{k})$ has the following form:
    \begin{align}\label{eq-formofchiu*}
        &q\left[\left((3\omega^{k-1}(\prod q_i)^{\frac{1}{3}}+{1\over 100}e^{\mathbf{i}\theta}\right)^3\prod q_i+ (-1)\cdot 27 \cdot(\prodl q_i)^2\right] +q^{\frac{4}{3}}Q_k, 
          \end{align}
   where $Q_k=Q_k(q^{\frac{1}{3}},q_1^{\frac{1}{3}},\dots,q_r^{\frac{1}{3}},e^{\mathbf{i}\theta})$ is a polynomial in $q^{\frac{1}{3}},q_1^{\frac{1}{3}},\dots,q_r^{\frac{1}{3}},e^{\mathbf{i}\theta}$. Then, there exists $\delta_{r*}>0$ so that when $|q|<\delta_{r*}$, $|q_i|\in[1,1+\delta_r]$ and $\theta\in\bR$, we have \[|q^{\frac{1}{3}}Q_k(q^{\frac{1}{3}},q_1^{\frac{1}{3}},\dots,q_r^{\frac{1}{3}},e^{\mathbf{i}\theta})|<{1\over 100}.\]  \par
   
   Now we argue by contradiction and assume $\chi_{M_r}(u^*_k)=0$. Then the above form \eqref{eq-formofchiu*} gives
   \begin{align*}
       \left|\left(3\omega^{k-1}(\prod q_i)^{\frac{1}{3}}+{e^{\mathbf{i}\theta}\over 100}\right)^3-  27\prod q_i\right|<\frac{1}{|\prod q_i|}\cdot{1\over 100}<{1\over 100},
   \end{align*}
   implying that
   \begin{align*}
       \left|3\cdot 3^2\omega^{2k-2}(\prod q_i)^{\frac{2}{3}}e^{\mathbf{i}\theta}+3\cdot 3\omega^{k-1}(\prod q_i)^{\frac{1}{3}}{e^{\mathbf{i}2\theta}\over 100}+{e^{\mathbf{i}3\theta}\over 100^2}\right|<1.
   \end{align*}
   Comparing the norms of the three terms, we see that the inequality is impossible. This contradiction proves that $\chi_{M_r}(u^*_k)\ne0$.
\end{proof}

\begin{lemma}\label{lemma-sharpestimate}
    Let $\delta_{r*}$ be as in the preceding lemma, and $\delta_r\in(0,{1\over 2}]$ as in Proposition \ref{prop-main-simple rightmost}. When $|q|\in (0,\delta_{r*}), |q_i|\in[1,1+\delta_r]$, we have
\begin{align*}
    |u_k(q;q_1,\dots,q_r)-3\omega^{k-1}q^{\frac{1}{3}}(\prod q_i)^{\frac{1}{3}}|<{1\over100}|q^{\frac{1}{3}}|,\quad k=1,2,3.
\end{align*}
\end{lemma}
\begin{proof}

We argue by contradiction. Assume that there exists $k_0\in\{1,2,3\}$ and $(q_0,q_{10},\dots,q_{r0})\in(\bC^\times)^{r+1}$ with $|q_0|\in(0,\delta_{r*}),|q_{i0}|\in[1,1+\delta_r]$, such that
\begin{equation*}
|u_{k_0}(q_0;q_{10},\dots,q_{r0})-3\omega^{k_0-1}q_0^{\frac{1}{3}}(\prod q_{i0})^{\frac{1}{3}}|\ge {1\over100}|q_0^{\frac{1}{3}}|.
\end{equation*}
For these fixed $q_{i0}$, it follows from Proposition \ref{theorem-puiseux series} that there exists $\delta_{r0}\in(0,|q_0|)$ such that when $|q'|\in (0,\delta_{r0})$, we have
\begin{equation*}
    |u_{k_0}(q';q_{10},\dots,q_{r0})-3\omega^{k_0-1}(q')^{\frac{1}{3}}(\prod q_{i0})^{\frac{1}{3}}|< {1\over100}|q'|^{\frac{1}{3}}.
\end{equation*}
By   continuity of the function $|u_{k_0}(q;q_{10},\dots,q_{r0})-3\omega^{k_0-1}q^{\frac{1}{3}}(\prod q_{i0})^{\frac{1}{3}}|-{1\over100}|q^{\frac{1}{3}}|$ in $q$,
we see that there exists $q'_0\in\bC^\times$ with $|q'_{0}|\in[\delta_{r0},|q_0|)\subset (0,\delta_{r*})$ such that 
\begin{equation*}
    |u_{k_0}(q'_{0};q_{10},\dots,q_{r0})-3\omega^{k_0-1}(q'_{0})^{\frac{1}{3}}(\prod q_{i0})^{\frac{1}{3}}|={1\over100}|q'_{0}|^{\frac{1}{3}},
\end{equation*}
contradicting Lemma \ref{prop- equality prop}. This proves the lemma.
\end{proof}

Let $\delta_r\in(0,{1\over 2}]$ be as in Proposition \ref{prop-main-simple rightmost}, and set $\bar{q}_i:=1+\frac{i}{r}\delta_r(1\le i\le r)$.

\begin{prop}\label{prop-estimateofeigevalueswhenrotating}
    There exists $\bar\delta_r\in(0,\delta_r)$, such that for any $q\in(0,\bar\delta_r)$ and $t\in[0,1]$:
    \begin{enumerate}
        \item the following inequalities hold:
        \begin{align*}
            |u_k(\rme^{2\pi\mathbf{i}t}q;\bar{q}_1,\dots,\bar{q}_r)-3\omega^{k-1}(\prod\bar{q}_i)^{\frac{1}{3}}(\rme^{2\pi\mathbf{i}t}q)^{\frac{1}{3}}|&<{1\over 100}|q|^{1\over 3},\quad k=1,2,3,\\
            |u_{3+k}(q;\bar{q}_1,\dots,\bar{q}_r)+\bar{q}_k|&<{1\over 100}\delta_r,\quad 1\le k\le r;
        \end{align*}
        \item for $k_0\in\{1,\dots,r\}$, the following inequalities hold:
        \begin{align*}
            |u_k(\rme^{-2\pi\mathbf{i}t}q;\bar{q}_1,\dots,\rme^{2\pi\mathbf{i}t}\bar{q}_{k_0},\dots,\bar{q}_r)-3\omega^{k-1}(\prod\bar{q}_i)^{\frac{1}{3}}q^{\frac{1}{3}}|&<{1\over 100}|q|^{1\over 3},\quad k=1,2,3,\\
            |u_{3+k}(\rme^{-2\pi\mathbf{i}t}q;\bar{q}_1,\dots,\rme^{2\pi\mathbf{i}t}\bar{q}_{k_0},\dots,\bar{q}_r)+\bar{q}_k|&<{1\over 100}\delta_r,\quad 1\le k\le r, k\ne k_0,\\
            |u_{3+k_0}(\rme^{-2\pi\mathbf{i}t}q;\bar{q}_1,\dots,\rme^{2\pi\mathbf{i}t}\bar{q}_{k_0},\dots,\bar{q}_r)+\rme^{2\pi\mathbf{i}t}\bar{q}_{k_0}|&<{1\over 100}\delta_r.
        \end{align*}
    \end{enumerate}
\end{prop}

 \begin{proof} 
   Part (1) follows immediately from Lemma \ref{lemma-sharpestimate} together with Proposition \ref{theorem-puiseux series}. 
    
    For case (2),  we have $u_{3+k_0}(0;\bar{q}_1,\dots,\rme^{2\pi\mathbf{i}t}\bar{q}_{k_0},\dots,\bar{q}_r)=-\rme^{2\pi\mathbf{i}t}\bar{q}_{k_0}$, and \[u_{3+k}(0;\bar{q}_1,\dots,\rme^{2\pi\mathbf{i}t}\bar{q}_{k_0},\dots,\bar{q}_r)=-\bar{q}_k\quad (1\le k\le r, k\ne k_0).
    \]
    Recall that entries of $M_r$ are in $\bR[q,q_1,\dots,q_r]$. So by the continuous dependence of roots on coefficients, for each $t\in[0,1]$, there exists $\delta_r(t)\in(0,{1\over 10000}\delta_r)$ such that, when $|q'|<\delta_r(t)$ and $t'\in[0,1]\cap(t-\delta_r(t),t+\delta_r(t))$, we have $|u_{3+k_0}(q';\bar{q}_1,\dots,\rme^{2\pi\mathbf{i}t'}\bar{q}_{k_0},\dots,\bar{q}_r)+\rme^{2\pi\mathbf{i}t}\bar{q}_{k_0}|<{1\over 10000}\delta_r$ and 
    \[
    |u_{3+k}(q';\bar{q}_1,\dots,\rme^{2\pi\mathbf{i}t}\bar{q}_{k_0},\dots,\bar{q}_r)+\bar{q}_k|<{1\over 100}\delta_r\quad (1\le k\le r, k\ne k_0),
    \]
    which implies  
    \begin{align*}
    &|u_{3+k_0}(q';\bar{q}_1,\dots,\rme^{2\pi\mathbf{i}t'}\bar{q}_{k_0},\dots,\bar{q}_r)+e^{2\pi\mathbf{i}t'}\bar{q}_{k_0}|\\\le{}& |u_{3+k_0}(q';\bar{q}_1,\dots,\rme^{2\pi\mathbf{i}t'}\bar{q}_{k_0},\dots,\bar{q}_r)+e^{2\pi\mathbf{i}t}\bar{q}_{k_0}|+|e^{2\pi\mathbf{i}t'}-e^{2\pi\mathbf{i}t}|\cdot|\bar{q}_{k_0}|
    <{1\over 100}\delta_r.
    \end{align*}
    So by compactness of $[0,1]$, we can find $\delta'_r\in(0,{1\over 10000}\delta_r)$ such that when $|q'|<\delta'_r$ and $t\in[0,1]$, we have $|u_{3+k_0}(q';\bar{q}_1,\dots,\rme^{2\pi\mathbf{i}t}\bar{q}_{k_0},\dots,\bar{q}_r)+\rme^{2\pi\mathbf{i}t}\bar{q}_{k_0}|<{1\over 100}\delta_r$ and 
    \[
    |u_{3+k}(q';\bar{q}_1,\dots,\rme^{2\pi\mathbf{i}t}\bar{q}_{k_0},\dots,\bar{q}_r)+\bar{q}_k|<{1\over 100}\delta_r\quad (1\le k\le r, k\ne k_0).
    \]
    Together with Lemma \ref{lemma-sharpestimate}, this proves case (2).
\end{proof}

 \subsection{Gamma conjecture II for $X_r$}
 We prove Theorem \ref{thm-GammaIIforXr} by verifying all assumptions of the Strategy-Theorem. 

Note that $X_r$ is a Fano manifold and $H^*(X_r)$ is generated by $H^2(X_r)$. So the big quantum cohomology of $X_r$ converges near the large radius limit by \cite[Corollary 5.9]{Iri07}.

 It was shown in \cite[Theorem 1.2]{HKLY21} that the original Gamma conjecture I of Galkin, Golyshev and Iritani holds for del Pezzo surfaces $X_r$. Hence, in the terminology of the present paper, we have  
\begin{prop}\label{prop: assumption1}
     $X_r$ satisfies property Gamma-I at $\mathbf{t}^{(2)}_0=\mathbf{0}\in H^2(X_r)$.
\end{prop}
\noindent We note that the  quantum cohomology of $X_r$ is non-semisimple  at $\mathbf{0}$  when $r>4$ \cite{BM04}.

Let  $\bar\delta_r$, $\bar{q}_1,\dots, \bar{q}_r$ be as in Proposition \ref{prop-estimateofeigevalueswhenrotating}, and choose  $\bar{q}\in(0, \min\{\bar\delta_r, 10^{-6}\})$. Define
\begin{equation}     \bar{\bft}^{(2)}:=\sum_{i=1}^r\log\bar{q}_i\cdot T_i+\log\bar{q}\cdot T_{r+1},
 \end{equation} 
Then $\bar{\bft}^{(2)}\in V_{\delta_r}$ and, in the present case, plays the role of $\bft^{(2)}_1$ in the Strategy-Theorem.
Denote by  $\bar{u}_{1},\dots,\bar{u}_{3+r}$   the eigenvalues of $\hat{c}_1(\bar{\bft}^{(2)})$.
\begin{prop}\label{prop: assumption2}
  There exists a simply connected domain  $U$ inside the ${\rm(SR)}$-region of $X_r$, such that  $V_{\delta_r}\subset U$ and $\bar{\bft}^{(2)}\in U\cap K_{ts}$.  
\end{prop}
\begin{proof}
    Let $U'$ denote the connected component of $\mathbf{0}$ in the ${\rm(SR)}$-region of $X_r$. The subset $V_{\delta_r}$ is connected by definition  and is contained in the
     ${\rm(SR)}$-region of $X_r$ by Proposition \ref{prop-main-simple rightmost}. Thus $V_{\delta_r}\subset U'$ since $\mathbf{0}\in V_{\delta_r}$.
  Since $V_{\delta_r}$ is simply connected, we can further find a simply connected domain $U$ such that $V_{\delta_r}\subset U\subset U'$. 
     
   By Proposition \ref{prop-estimateofeigevalueswhenrotating}, the eigenvalues $\bar u_i$ are pairwise distinct from each other; they are depicted  in Figure \ref{figure1}. Thus $\bar\bft^{(2)}\in K_{ts}$, and hence $\bar\bft^{(2)}\in V_{\delta_r}\cap K_{ts}\subset U\cap K_{ts}$.     
\end{proof} 
   
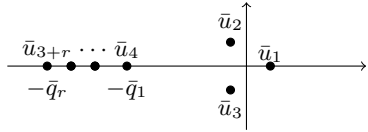
\begin{figure}[htbp]
\centering
\begin{tikzpicture}[x=30pt,y=30pt]
\draw[->] (-2,0) -- (2.5,0);
\draw[->] (1,-0.8) -- (1,0.8);

\filldraw (1.3, 0) circle [radius =0.05];
\draw (1.3,0.15) node {\small $\bar{u}_{1}$};

\filldraw (0.8, 0.3) circle [radius =0.05];
\draw (0.8,0.55) node {\small $\bar{u}_{2}$};

\filldraw (0.8, -0.3) circle [radius =0.05];
\draw (0.8,-0.55) node {\small $\bar{u}_{3}$};

\filldraw (-0.5, 0) circle [radius =0.05];
\draw (-0.5,0.2) node {\small $\bar{u}_{4}$};
\draw (-0.5,-0.3) node {\small $-\bar{q}_1$};

\filldraw (-0.9, 0) circle [radius =0.05];
\filldraw (-0.9, 0) circle [radius =0.05];

\filldraw (-1.2, 0) circle [radius =0.05];
\filldraw (-1.2, 0) circle [radius =0.05];
\draw (-0.9,0.2) node {\small $\cdots$};

\filldraw (-1.5, 0) circle [radius =0.05];
\draw (-1.5,0.2) node {\small $\bar{u}_{3+r}$};
\draw (-1.5,-0.3) node {\small $-\bar{q}_r$};

\end{tikzpicture}
\vspace{-1em}
\caption{Eigenvalues $\bar{u}_{1},\dots,\bar{u}_{3+r}$ of $\hat{c}_1(\bar{\bft}^{(2)})$.}
\label{figure1}
\end{figure}

We will use the ordering of $[\bar u_1, \dots, \bar u_{3+r}]$ as in Figure  \ref{figure1} in the rest of this subsection.

Recall that  we have abused notation by identifying an object of $\mathcal{D}^b_{\rm coh}(X)$ with its corresponding class in $K(X)$.

\begin{lemm}\label{lem-determine coefficients}
Let $(F_1, \dots, F_s)$ be a full exceptional collection in $\mathcal{D}^b_{\rm coh}(X)$. Suppose that  $V \in K(X)\otimes_\mathbb{Z} \mathbb{C}$   satisfies
\begin{equation*}
 F_{j_1}+mF_{j_2} = \chi\left(F_{j_1} + F_{j_2}, V\right) \cdot V
\end{equation*}
for some $m\in\mathbb{Z}$ and $j_1 < j_2$. Assume further that $\chi(F_{j_1}, F_{j_2}) = -1$. Then $V = \pm (F_{j_1}+mF_{j_2})$.
\end{lemm}
\begin{proof}
   Since $F_1, \dots, F_s$ form a $\bZ$-basis for $K(X)$, it follows that we can write $V = \sum_{i=1}^s a_iF_i$. Let $\xi := \chi(F_{j_1} + F_{j_2}, V)$. Then the condition becomes $F_{j_1}+mF_{j_2} = \xi \sum a_iF_i$. Comparing coefficients immediately yields 
    \begin{equation*}
    \xi \cdot a_{j_1} = 1,\,\, \xi \cdot a_{j_2} = m, \text{ and } \xi\cdot a_i = 0\text{ for all } i \neq j_1,j_2.    
   \end{equation*}
This forces $a_i = 0$ when $i \neq j_1,j_2$, yielding $V = a_{j_1}F_{j_1}+a_{j_2}F_{j_2}$. Evaluating $\xi$ via bilinearity gives
\begin{align*}
   \xi =a_{j_1}\chi(F_{j_1},F_{j_1})+a_{j_1}\chi(F_{j_2},F_{j_2})+a_{j_2}\chi(F_{j_1},F_{j_2})+a_{j_2}\chi(F_{j_2},F_{j_2})= a_{j_1}.
\end{align*}
Here in the second euality, we used the definition of an exceptional collection and the assumption $\chi(F_{j_1},F_{j_2})=-1$.  Combining $\xi=a_{j_1}$ with the relation $\xi \cdot a_{j_1} = 1$ and $\xi \cdot a_{j_2} = m$, we obtain  $a_{j_1} = \pm 1$ and correspondingly $a_{j_2}=\pm m$. This proves the lemma.
\end{proof}

Consider the paths $\phi_+(t),\phi_-(t),\phi_k(t)(1\le k\le r):[0,1]\to H^2(X_r)$ given by 
    \begin{equation*}
    \phi_\pm(t):=\bar{\bft}^{(2)}\pm2\pi\mathbf{i} t\cdot H, \,\text{ and }\, \phi_k(t):=\bar{\bft}^{(2)}-2\pi\mathbf{i}t\cdot E_k.
    \end{equation*} 
  We note that  the images of these paths are contained in $K_{ts}$ by Proposition \ref{prop-estimateofeigevalueswhenrotating}.

\begin{prop}\label{prop: assumption3a}
    Denote $V_1:=\mathcal{O}$. Let $V_2, \cdots, V_{3+r}$ be the objects in $D^b_{\rm coh}(X_r)$, such that they are obtained from $\cO$ via Galois actions given by the paths  $\phi_+(t),\phi_-(t),\phi_1(t),\dots,  \phi_r(t)$.  Let $\tilde{V}_1=V_3,\,\tilde{V}_2=V_1=\cO,\,\tilde{V}_3=V_2$. For each $1\leq k\leq r$, let $\tilde{V}_{3+k}\in K(X)\otimes_\bZ\bC$ satisfy the equation
   \begin{equation*}
       \mathcal{Z}^{K,\rm{sm}}(V_2)= \mathcal{Z}^{K,\rm{sm}}(V_{3+k})-[\mathcal{Z}^{K,\rm{sm}}(V_{3+k}),\mathcal{Z}^{K,\rm{sm}}(\tilde{V}_{3+k}))\cdot\mathcal{Z}^{K,\rm{sm}}(\tilde{V}_{3+k}).
   \end{equation*} Then 
   $\tilde{V}_{3+k}=\pm L_{\mathcal{O}(-H)}\mathcal{O}(E_k-H)$. In particular,   
   $ (\tilde V_{3+r}, \dots, \tilde V_4, \tilde V_3, \tilde V_2, \tilde V_1)$ form a full exceptional collection of $D^b_{\rm coh}(X_r)$.
\end{prop}

\begin{proof}
    It follows from  Proposition \ref{prop: assumption2} and Theorem \ref{thm-global gamma-I} that $X_r$ satisfies Gamma-I at $\bar{\bft}^{(2)}$. Hence $\mathcal{Z}^{K,{\rm{sm}}}(\cO)$ can be  represented at $\bar{\bft}^{(2)}$ by the half line $L_1:=\bar{u}_{1}+\bR_{\ge0}\rme^{\mathbf{i}0}$. 
    
    Consider the movement of eigenvalues given by the path $\phi_+(t)$. Denote the eigenvalues of $\hat{c}_1(\phi_+(t))$ by $\{u^+_{i}(t)\}_i$ with $u^+_{i}(0)=\bar{u}_{i}$. By Proposition \ref{prop-estimateofeigevalueswhenrotating} (1), we see that $u^+_{1}(t),u^+_{2}(t),u^+_{3}(t)$ move almost along the circle centered at origin with radius $3(\prod\bar{q}_i)^{1\over 3}\bar{q}^{1\over 3}$  counterclockwise,  and they do not collide with each other. In addition, $u^+_{3+k}(t)$  is  nearly fixed and approximately equal to $-\bar{q}_k$ for $1\le k\le r$. As $u^+_1(t)$ moves, the corresponding translation of $L_1$ does not pass through other eigenvalues $u^+_i(t)$ ($i\ne1$), and at $t=1$,  $L_1$ is  translated to  $L_2:=\bar{u}_{2}+\bR_{\ge0}\rme^{\mathbf{i}0}$ since $u^+_1(1)=\bar{u}_2$. So $\mathcal{Z}^{K,{\rm{sm}}}(\cO)$ can be represented by $L_2$ at $\phi_+(1)=\bar{\bft}^{(2)}+2\pi\mathbf{i}H$. From the Galois action \eqref{eqn-Galoisaction} and its $K$-group framing interpretation \eqref{equ-monodromy-invariance}, this implies that $\mathcal{Z}^{K,{\rm{sm}}}(\cO(-H))=G(-H)(\mathcal{Z}^{K,{\rm{sm}}}(\cO))$ can be represented by $L_2$ at $\bar{\bft}^{(2)}$. Thus $\cO(-H)=V_2=\tilde{V}_3$. Similarly, considering movement given by the path $\phi_-(t)$, we see that $\mathcal{Z}^{K,{\rm{sm}}}(\cO(H))$ can be represented by the half line $L_3:=\bar{u}_{3}+\bR_{\ge0}\rme^{\mathbf{i}0}$ at $\bar{\bft}^{(2)}$, and then $\cO(H)=V_3=\tilde{V}_1$.

For $k_0=1,\dots,r$, we consider the movement of eigenvalues given by the path $\phi_{k_0}(t)$. Denote the eigenvalues of $\hat{c}_1(\phi_{k_0}(t))$ by $\{u^{(k_0)}_{i}(t)\}_i$ with $u^{(k_0)}_{i}(0)=\bar{u}_{i}$. By Proposition \ref{prop-estimateofeigevalueswhenrotating} (2), we see that for $1\le j\le 3$, $u^{(k_0)}_{j}(t)$ is  nearly fixed and approximately equal to $3\omega^{j-1}(\prod\bar{q}_i)^{1\over 3}\bar{q}^{1\over 3}$, and for $1\le k\le r$ with $k\ne k_0$, $u^{(k_0)}_{3+k}(t)$ is  nearly fixed and approximately equal to $-\bar{q}_k$. Besides, $u^{(k_0)}_{3+k_0}(t)$ moves almost along the circle centered at origin with radius $\bar{q}_{k_0}$ counterclockwise. As $u^{(k_0)}_{3+k_0}(t)$ moves, we correspondingly rotate and deform $L_1$ so that it does not pass through eigenvalues $u^{(k_0)}_i(t)$ for $i\ne2$, and at $t=1$ we get the (heavily bent) admissible path $L'_{3+k_0}$ which starts from $\bar{u}_2$, goes around $\bar{u}_{3+k_0}$ and has zero end-phase (see Figure \ref{figure3}).  So $\mathcal{Z}^{K,{\rm{sm}}}(\cO(-H))$ can be represented by $L'_{3+k_0}$ at $\phi_{k_0}(1)=\bar{\bft}^{(2)}-2\pi\mathbf{i}E_{k_0}$. From the Galois action \eqref{eqn-Galoisaction} and its $K$-group framing interpretation \eqref{equ-monodromy-invariance}, this implies that $\mathcal{Z}^{K,{\rm{sm}}}(\cO(E_{k_0}-H))=G(E_{k_0})(\mathcal{Z}^{K,{\rm{sm}}}(\cO(-H)))$ can be represented by $L'_{3+k_0}$ at $\bar{\bft}^{(2)}$. Then $V_{3+k_0}=\cO(E_{k_0}-H)$.

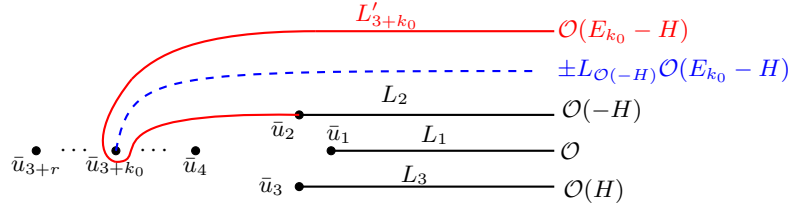
\begin{figure}[htbp]
\vspace{-1em}
\centering
\begin{tikzpicture}[x=30pt,y=30pt]
\draw (0,0) node {\phantom{A}};

\filldraw (5.2, 0) circle [radius =0.05];
\draw (5.3,0.2) node {\small $\bar{u}_1$};

\filldraw (4.8, 0.45) circle [radius =0.05];
\draw (4.6,0.25) node {\small $\bar{u}_2$};

\filldraw (4.8, -0.45) circle [radius =0.05];
\draw (4.45,-0.45) node {\small $\bar{u}_3$};

\filldraw (3.5, 0) circle [radius =0.05];
\draw (3.5,-0.2) node {\small $\bar{u}_4$};

\draw (3,0) node {\small $\cdots$};

\filldraw (2.5, 0) circle [radius =0.05];
\draw (2.5,-0.2) node {\small $\bar{u}_{3+k_0}$};

\draw (2,0) node {\small $\cdots$};

\filldraw (1.5, 0) circle [radius =0.05];
\draw (1.5,-0.2) node {\small $\bar{u}_{3+r}$};

\draw[thick] (5.2,0) -- (8,0);
\draw[thick] (4.8,0.45) -- (8,0.45);

\draw[thick] (4.8,-0.45) -- (8,-0.45);

\draw[color=red,thick] (4.8,0.45) .. controls  (2,0.5) and (3,-0.14)  .. (2.5,-0.14);
\draw[color=red,thick] (2.5,-0.14) .. controls  (2.3,-0.1) and (2.2,0.28) .. (2.7,0.9);
\draw[color=red,thick] (2.7,0.9) .. controls  (3,1.2) and (3.5,1.5) .. (5,1.5);

\draw[color=red,thick] (5,1.5) .. controls  (7,1.5) .. (8,1.5);
\draw[color=red] (8.87,1.5) node {\small $\cO(E_{k_0}-H)$};
\draw[color=red] (5.9,1.7) node {\small$L_{3+k_0}'$};

\draw[color=blue,thick,dashed] (2.5,0) .. controls  (2.7,1) and (3,1) .. (7.75,1);
\draw[color=blue] (9.5,1) node {\small $\pm L_{\cO(-H)}\cO(E_{k_0}-H)$};

\draw (8.2,0) node {\small$\cO$};
\draw (6.5,0.2) node {\small$L_1$};

\draw (8.6,0.5) node {\small$\cO(-H)$};

\draw (6,0.7) node {\small$L_2$};

\draw (8.5,-0.5) node {\small$\cO(H)$};
\draw (6.25,-0.3) node {\small$L_3$};

\end{tikzpicture}
\vspace{-1em}
\caption{Admissible paths and corresponding $K$-classes. All paths have zero end-phase, and solid paths are given by monodromy transformation from $L_1$. }
\label{figure3}
\end{figure}

 Take the simply connected domain $U$ as in Proposition \ref{prop: assumption2}. Then $U\subset K_{ts}$, and suppose that $U$  is properly-chosen with respect to $\{(u_i,\Psi_i)\}$, where $u_i(\bar\bft^{(2)})=\bar u_i$ for all $i$. Let $\hat{y}_i$ be the $\widehat{\nabla}^{\rm{sm}}$-flat section from Lemma \ref{lemma-holomorphicflatsectionwrtlaplacedualconn}. Choosing suitable signs of $\Psi_1,\Psi_2,\Psi_3$, we can suppose that the Laplace transforms of $\hat{y}_1,\hat{y}_2,\hat{y}_3$ over $L_1,L_2,L_3$ give $\mathcal{Z}^{K,{\rm{sm}}}(\cO)$, $\mathcal{Z}^{K,{\rm{sm}}}(\cO(-H))$, $ \mathcal{Z}^{K,{\rm{sm}}}(\cO(H))$, respectively. Let $y'_{3+k_0}$ be the Laplace transform of $\hat{y}_{3}$ over $L'_{3+k_0}$. Then we have a priory $y'_{3+k_0}=\pm\mathcal{Z}^{K,{\rm{sm}}}(\cO(E_{k_0}-H))$ by the sign ambiguity of normalized idempotents. We will show that $y'_{3+k_0}=\mathcal{Z}^{K,{\rm{sm}}}(\cO(E_{k_0}-H))$ in \eqref{equ-positive determination of O(E-H)}. Construct a  dotted curve $L''_{3+k_0}$ as in Figure \ref{figure3} that starts from $\bar{u}_{3+k_0}$, does not touch $L_2$, $L_{3+k_0}'$ nor  eigenvalues $\bar u_i (i\neq 3+k_0)$, and has zero end-phase.
Let $y_{3+k_0}$ be the Laplace transform of $\hat{y}_{3+k_0}$ over the dotted curve $L''_{3+k_0}$. Then $\mathcal{Z}^{K,{\rm{sm}}}(\cO(-H))$ is the right mutation of $y'_{3+k_0}$ with respect to $\bar{u}_{3+k_0}$, implying
\begin{align}\label{eqn-rightmutationeqn}
    \mathcal{Z}^{K,{\rm{sm}}}(\cO(-H))=y'_{3+k_0}-[y'_{3+k_0},y_{3+k_0})\cdot y_{3+k_0}.
\end{align} 
We claim that for any $k_0\in\{1, \dots, r\}$, $y_{3+k_0}=\pm\mathcal{Z}^{K,\rm{sm}}(L_{\cO(-H)}\cO(E_{k_0}-H))$; that is, $\tilde{V}_{3+k_0}=L_{\cO(-H)}\cO(E_{k_0}-H)$.  
Note that $\cO(-H)$, $\cO$, $\cO(H)$ form a full exceptional collection on $\mathcal{D}^b_{\rm coh}(\bP^2)$. Then their \emph{augmentation}, $\cO(-H)$, $\cO(E_r-H),\dots,\cO(E_1-H),\cO,\cO(H)$, form a full exceptional collection of $\mathcal{D}^b_{\rm coh}(X_r)$ (cf. \cite{HP11}). Applying left mutations, we see that $$L_{\cO(-H)}\cO(E_r-H),\dots,L_{\cO(-H)}\cO(E_1-H),\cO(-H),\cO,\cO(H)$$ also form a full exceptional collection. That is, $(\tilde{V}_{3+r},\dots,\tilde{V}_1)$ form a full exceptional collection.

It remains to show our claim. Let $\tilde{V}_{3+k_0}\in K(X_r)\otimes_{\mathbb{Z}}\mathbb{C}$ satisfy $y_{3+k_0}=\mathcal{Z}^{K,\rm{sm}}(\tilde{V}_{3+k_0})$. Note that $y'_{3+k_0}=\pm\mathcal{Z}^{K,{\rm{sm}}}(\cO(E_{k_0}-H))$. We first assume that $y'_{3+k_0} = -\mathcal{Z}^{K,{\rm{sm}}}(\cO(E_{k_0}-H))$, and then \eqref{eqn-rightmutationeqn} together with \eqref{equ-resp pairing} gives 
\begin{align}\label{eqn-negativesignrelationfory'3+k0}
\cO(-H)=-\left(\cO(E_{k_0}-H)-\chi(\cO(E_{k_0}-H),\tilde{V}_{3+k_0})\cdot \tilde{V}_{3+k_0}\right).
\end{align}
Recall the following relation in $K(X)$:
\begin{align}\label{eqn-leftmutationrelationinKX}
L_{\cO(-H)}\cO(E_{k_0}-H) = \cO(E_{k_0}-H) - \cO(-H).
\end{align}
Then  
\begin{equation*}
\chi(L_{\cO(-H)}\cO(E_{k_0}-H),\cO(-H))\stackrel{\text{\eqref{eqn-leftmutationrelationinKX}}}{=}\chi(\cO(E_{k_0}-H)-\cO(-H),\cO(-H))=-1,
\end{equation*}
where definition of exceptional collection is used in the second equality, and
\begin{align*}     \cO(-H)+\cO(E_{k_0}-H)&\stackrel{\text{\eqref{eqn-leftmutationrelationinKX}}}{=}L_{\cO(-H)}\cO(E_{k_0}-H)+2\cO(-H)\\
    &\stackrel{\eqref{eqn-negativesignrelationfory'3+k0}}{=}\chi(L_{\cO(-H)}\cO(E_{k_0}-H)+\cO(-H),\tilde{V}_{3+k_0})\cdot \tilde{V}_{3+k_0}.
\end{align*}
Apply Lemma \ref{lem-determine coefficients} with $F_{j_1}=L_{\cO(-H)}\cO(E_{k_0}-H),F_{j_2}=\cO(-H),m=2$, we get $\tilde{V}_{3+k_0}=\pm\left(L_{\cO(-H)}\cO(E_{k_0}-H)+2\cO(-H)\right)$. \par
However, as part of AEFS near $\bar\bft^{(2)}$ associated with  a sufficiently small   positive phase, we must have $[y_2,y_{3+k_0})=0$; on the other hand, $$[y_2,y_{3+k_0})\stackrel{\text{\eqref{equ-resp pairing}}}{=}\chi(\cO(-H),L_{\cO(-H)}\cO(E_{k_0}-H)+2\cO(-H))=2,$$where definition of exceptional collection is used in the second equality. We get a contradiction. So we conclude that 
\begin{equation}\label{equ-positive determination of O(E-H)}
    y'_{3+k_0}=\mathcal{Z}^{K,{\rm{sm}}}(\cO(E_{k_0}-H))
\end{equation}
A similar application of Lemma \ref{lem-determine coefficients} gives $\tilde{V}_{3+k_0}=\pm L_{\cO(-H)}\cO(E_{k_0}-H)$. \qedhere

\begin{figure}[htbp]
\vspace{-1em}
\centering
\begin{tikzpicture}[x=30pt,y=30pt]
\draw (0,0) node {\phantom{A}};

\filldraw (5.2, 0) circle [radius =0.05];
\draw (5.3,0.2) node {\small $\bar{u}_1$};

\filldraw (4.8, 0.45) circle [radius =0.05];
\draw (4.6,0.25) node {\small $\bar{u}_2$};

\filldraw (4.8, -0.45) circle [radius =0.05];
\draw (4.45,-0.45) node {\small $\bar{u}_3$};

\filldraw (3.5, 0) circle [radius =0.05];
\draw (3.5,-0.2) node {\small $\bar{u}_4$};

\draw (3,0) node {\small $\cdots$};

\filldraw (2.5, 0) circle [radius =0.05];
\draw (2.5,-0.2) node {\small $\bar{u}_{3+k_0}$};

\draw (2,0) node {\small $\cdots$};

\filldraw (1.5, 0) circle [radius =0.05];
\draw (1.5,-0.2) node {\small $\bar{u}_{3+r}$};

\draw[thick] (5.2,0) -- (8,0);
\draw[thick] (4.8,0.45) -- (8,0.45);

\draw[thick] (4.8,-0.45) -- (8,-0.45);

\draw[thick] (3.5,0) .. controls  (3.7,0.85) and (4,0.85) .. (8,0.85);
\draw (9.4,0.85) node {\small $\pm L_{\cO(-H)}\cO(E_1-H)$};
\draw (3.35,0.5) node {\small $L_4''$};

\draw[thick] (2.5,0) .. controls  (2.7,1.3) and (3,1.3) .. (8,1.3);
\draw (6.5,1.17) node {\small $\vdots$};
\draw (9.5,1.3) node {\small $\pm L_{\cO(-H)}\cO(E_{k_0}-H)$};
\draw (2.15,0.5) node {\small $L_{3+k_0}''$};

\draw[thick] (1.5,0) .. controls  (1.7,1.8) and (2,1.8) .. (8,1.8);
\draw (6.3,1.65) node {\small $\vdots$};
\draw (9.5,1.8) node {\small $\pm L_{\cO(-H)}\cO(E_{r}-H)$};
\draw (1.25,0.5) node {\small $L_{r}''$};

\draw (8.2,0) node {\small$\cO$};
\draw (6.5,0.18) node {\small$L_1$};

\draw (8.6,0.5) node {\small$\cO(-H)$};

\draw (6,0.65) node {\small$L_2$};

\draw (8.5,-0.5) node {\small$\cO(H)$};
\draw (6.25,-0.3) node {\small$L_3$};

\end{tikzpicture}
\vspace{-1em}
\caption{Admissible paths and corresponding $K$-classes. All paths have zero end-phase.}
\label{figure4}
\end{figure}
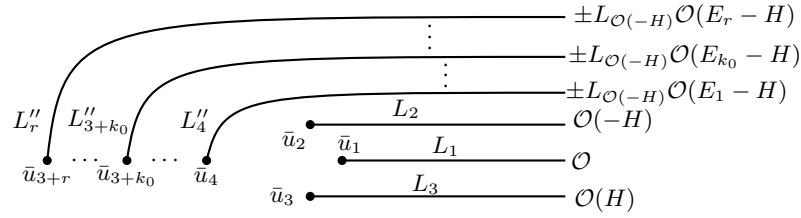

\end{proof}

\begin{prop}\label{prop: assumption3b}
  Let $(\tilde V_{3+r}, \dots, \tilde V_{1})$ be given as in Proposition \ref{prop: assumption3a}. Set $\phi:={\pi\over 8}$. 
    Then  $\Im(\rme^{-\mathbf{i}\phi}\bar{u}_{r+3})>\dots>\Im(\rme^{-\mathbf{i}\phi}\bar{u}_{4})>\Im(\rme^{-\mathbf{i}\phi}\bar{u}_{2})>\Im(\rme^{-\mathbf{i}\phi}\bar{u}_{1})>\Im(\rme^{-\mathbf{i}\phi}\bar{u}_{3})$, and   $\mathcal{Z}^{K, \mathrm{sm}}(\tilde{V}_{i})$ respects  $(u_i, \Psi_i)$ for $4\leq i\leq r+3$, $\mathcal{Z}^{K, \mathrm{sm}}(\tilde{V}_{3})$ respects $(u_2, \Psi_2)$, $\mathcal{Z}^{K, \mathrm{sm}}(\tilde{V}_{2})$ respects $(u_1, \Psi_1)$, and $\mathcal{Z}^{K, \mathrm{sm}}(\tilde{V}_{1})$ respects $(u_3, \Psi_3)$   at $\bar{\bft}^{(2)}$ with  phase $\phi$.
\end{prop}

\begin{proof}

For $k=1,\dots,r$, the Laplace transform of $\hat{y}_{3+k}$ over the curved path $L''_{3+k}$ in Figure \ref{figure4} starting from $\bar{u}_{3+k}$ gives the flat section $\mathcal{Z}^{K,{\rm{sm}}}(\tilde{V}_{3+k})$, and similarly $\mathcal{Z}^{K,{\rm{sm}}}(\cO)=\mathcal{Z}^{K,{\rm{sm}}}(\tilde{V}_{2})$, $\mathcal{Z}^{K,{\rm{sm}}}(\cO(-H))=\mathcal{Z}^{K,{\rm{sm}}}(\tilde{V}_{3})$, $\mathcal{Z}^{K,{\rm{sm}}}(\cO(H))=\mathcal{Z}^{K,{\rm{sm}}}(\tilde{V}_{1})$ for $L_1$, $L_2$, $L_3$ respectively.\par
Now we  deform the paths $L_1,L_2,L_3,L''_4,\dots,L''_{3+r}$ in Figure \ref{figure4} to the parallel half lines in Figure \ref{figure5}, so that  each path has argument    $\phi$.  For $1\leq k\leq r-1$, we have
$(\bar{u}_{4+k}-\bar{u}_{3+k})=-{1\over r}\delta_r+(\bar{u}_{4+k}+\bar{q}_{k+1})-(\bar{u}_{3+k}+\bar{q}_{k})$, and consequently
 $\Im(\rme^{-\mathbf{i}\phi}(\bar{u}_{4+k}-\bar{u}_{3+k}))={\delta_r\over r}\sin\phi+ \Im( (\bar{u}_{4+k}+\bar q_{k+1})-(\bar{u}_{3+k}+\bar q_{k}))>{\delta_r\over r}\sin\phi-{\delta_r\over 50}>0$. Here the first inequality follows from Proposition \ref{prop-estimateofeigevalueswhenrotating} (1), and the second inequality follows by noting  $\phi={\pi \over 8}$. 
The  remaining  inequalities $\Im(\rme^{-\mathbf{i}\phi}\bar{u}_{4})>\Im(\rme^{-\mathbf{i}\phi}\bar{u}_{2})>\Im(\rme^{-\mathbf{i}\phi}\bar{u}_{1})>\Im(\rme^{-\mathbf{i}\phi}\bar{u}_{3})$  follow from elementary estimates using Proposition \ref{prop-estimateofeigevalueswhenrotating} (1), since $\bar q<10^{-6}$ is sufficiently small.

Following the discussion between Remark \ref{rmk39} and Theorem \ref{thm:strategy}, the fact that the integration paths in Figure \ref{figure5} are parallel half-lines ensures that $\mathcal{Z}^{K,{\rm{sm}}}(\tilde{V}_{3+r}), \dots, \mathcal{Z}^{K,{\rm{sm}}}(\tilde{V}_{1})$ form an AEFS, and the statement follows.
\end{proof}
\begin{figure}[!htbp]
\vspace{-1.5em}
\centering
\begin{tikzpicture}[x=30pt,y=30pt]
\draw (0,0) node {\phantom{A}};

\filldraw (5.7, -0.3) circle [radius =0.05];
\draw (5.8,-0.15) node {\small $\bar{u}_1$};

\filldraw (5.3, 0) circle [radius =0.05];
\draw (5.1,-0.19) node {\small $\bar{u}_2$};

\filldraw (5.3, -0.6) circle [radius =0.05];
\draw (4.95,-0.6) node {\small $\bar{u}_3$};

\filldraw (3, -0.3) circle [radius =0.05];
\draw (3,-0.5) node {\small $\bar{u}_4$};

\draw (2.5,-0.3) node {\small $\cdots$};

\draw (2,-0.3) node {\small $\cdots$};

\filldraw (1.5, -0.3) circle [radius =0.05];
\draw (1.5,-0.5) node {\small $\bar{u}_{3+r}$};

\draw[thick] (5.7,-0.3) -- (7,-0.04);
\draw[thick] (5.3,0) -- (7,0.34);

\draw[thick] (5.3,-0.6) -- (7,-0.26);

\draw[thick] (3,-0.3) -- (7,0.5);

\draw[thick] (1.5,-0.3) -- (7,0.8);

\draw (7.25,0) node {\small$\cO$};
\draw (7.6,0.26) node {\small$\cO(-H)$};

\draw (7.48,-0.34) node {\small$\cO(H)$};
\draw (8.6,0.58) node {\small$\pm L_{\cO(-H)}\cO(E_1-H)$};
\draw (8,0.75) node {\small$\cdots\cdots$};
\draw (8.6,1) node {\small$\pm L_{\cO(-H)}\cO(E_r-H)$};
\end{tikzpicture}
\vspace{-1em}
\caption{Admissible paths and corresponding $K$-classes. All paths are lines parallel to each other, with end-phase $\phi>0$ small.}
\label{figure5}
\end{figure}
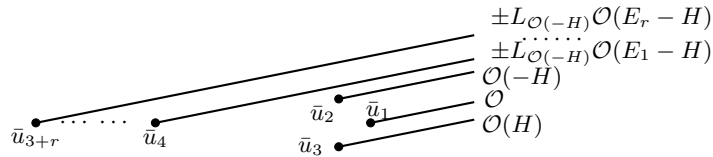

\begin{proof}[Proof of Theorem \ref{thm-GammaIIforXr}]
    
  Propositions \ref{prop: assumption1}, \ref{prop: assumption2}, \ref{prop: assumption3a} and \ref{prop: assumption3b} verify that  all assumptions of Theorem \ref{thm:strategy} hold. Hence, Gamma conjecture II holds for $X_r$ by Theorem  \ref{thm:strategy}.  
\end{proof}

\subsection{Corollary for  blow-ups of $\bP^2$}

Let $S_r$ be an $r$-fold blow-up of $\bP^2$, i.e., $S_r$ is obtained by $S_r\xrightarrow{b_r}S_{r-1}\xrightarrow{b_{r-1}}\cdots\xrightarrow{b_2}S_1\xrightarrow{b_1}\bP^2$, where each $b_i$ is the blow-up at one point. In general, $S_r$ need not be Fano.
\begin{thm}
    Gamma conjecture II holds for $S_r$ with $1\le r\le 8$.
\end{thm}
\begin{proof}
    Note that $S_r$  can be deformed to $X_r$. By deformation invariance of Gromov-Witten invariants, Theorem \ref{thm-GammaIIforXr} implies that, fixing $\bft_0\in B_{ss}(S_r)$ and admissible phase $\theta\in\bR$ at $\bft_0$, there exist  $V_1,\dots,V_s\in K(S_r)$, such that  $\{\mathcal{Z}^K(V_i)\}$ is an AEFS near $\bft_0$ associated to phase $\theta$,  and $\{V_i\}$ is  {an exceptional basis} in the pseudolattice $(K(S_r),\chi)$, i.e., it is a $\bZ$-basis of $K(S_r)$ with $\chi(V_i,V_i)=1$ and $\chi(V_j,V_i)=0$ whenever $i<j$.  Note that $\mathcal{D}^b(S_r)$ admits full exceptional collections by Orlov, and let $\{E_i\}$ be such a collection.  By \cite[Theorem 1.2 (ii)]{Kra24},  the two exceptional bases $\{V_i\}$ and $\{E_i\}$ in $(K(S_r),\chi)$ are related by mutations and sign changes, implying that $\{V_i\}$ arises from a full exceptional collection in $\mathcal{D}^b(S_r)$. This finishes the proof.
\end{proof}

\section{Condition ${\rm (SR)}$ for del Pezzo surfaces}\label{section-condition*forXr}

This section is devoted to prove Proposition \ref{prop-main-simple rightmost}. 

\subsection{Condition $(\mathrm{SR})$ via Perron--Frobenius theorem}\label{subsection-condition*forXr}
A nonnegative square matrix $M$ is \textbf{primitive} if $M^m$ is positive for some $m\in \mathbb{Z}_{\geq 1}$ (see \cite[Definition 1.1]{Sen81}). We use the following standard form of the Perron--Frobenius theorem (see \cite[Theorem 1.1]{Sen81}). 
\begin{prop}\label{prop-PFforprimitive}
    Let $M$ be a primitive matrix. Then $M$ has a positive simple eigenvalue $u$ such that for any other eigenvalue $u'$, we have $u>|u'|$.
    
    %Then the spectral radius 
    %\[\max\{|u| \,\big|\, u\in\bC \text{ is an eigenvalue of }A\}\] 
    %itself is an eigenvalue of $A$ with  multiplicity one.
\end{prop}

To prove Proposition \ref{prop-main-simple rightmost}, we consider the matrix   $M^{(a)}_r=\left((M_r^{(a)})_{ij}\right)_{1\le i,j\le r+3}$   of $\hat{c}_1(\bft^{(2)})$ with respect to the ordered basis $[\mathbf{1},aH-\sum_iE_i,E_1,...,E_r,[pt]]$, where $a>0$. 
The dual basis with respect to Poincar\'e pairing is given by
\[
\mathbf{1}^\vee=[pt],\quad (aH-\sum_iE_i)^\vee={H\over a},\quad E_i^\vee={H\over a}-E_i,\quad [pt]^\vee=\mathbf{1}.
\]
Then each entry $(M_r^{(a)})_{ij}$ is a polynomial in $q,q_1,...,q_r$, and can be read off as a summation of three-point Gromov-Witten invariants over $\Eff(X_r)$. For instance, 
\[
   (M_r^{(a)})_{3 2}=\sum_{\bfd} \langle c_1,E_1, (aH-\sum_iE_i)^\vee\rangle_\bfd^{X_r}\cdot\bfq^{\bfd}.\quad(\text{see \eqref{notation-q^d} for }\bfq^{\bfd})
\]

\begin{exam}\label{example-r=2}
By Example \ref{example-r=1} and a change of bases, we have 
   \begin{equation*}
M_1^{(a)}=\begin{bmatrix}
            0& (2a-2)q& 2q&3qq_1\\
            \frac{3}{a}& 0& 0 &\frac{2q}{a}\\
            -1+\frac{3}{a}& q_1& -q_1 & -2q+\frac{2q}{a}\\
            0& -1+3a& 1& 0
            
        \end{bmatrix},
    \end{equation*} 
and the matrix $M_2^{(a)}$ is given by
\begin{align*}
\begin{bmatrix}
            0& (2a-2)q(q_1+q_2)& 2qq_2& 2qq_1&3qq_1q_2\\
            \frac{3}{a}& q-\frac{2q}{a}& \frac{q}{a}& \frac{q}{a} &2\frac{q}{a}(q_1+q_2)\\[0.1cm]
            -1+\frac{3}{a}& (2-a)(q-\frac{q}{a})+q_1& -q+\frac{q}{a}-q_1& -q+\frac{q}{a} & -2qq_2+2\frac{q}{a}(q_1+q_2)\\[0.1cm]
            -1+\frac{3}{a}& (2-a)(q-\frac{q}{a})+q_2& -q+\frac{q}{a}& -q+\frac{q}{a}-q_2&-2qq_1+2\frac{q}{a}(q_1+q_2)\\[0.1cm]
            0& -2+3a& 1& 1&0
        \end{bmatrix}.
\end{align*}
    
\end{exam}

\par

We first assume the following proposition, and leave its proof to next subsection.
\begin{prop}\label{prop-all positive entries but 1}
For $3\le r\le 8$, there exist $\delta_r\in (0,\frac{1}{2}]$, $a_r>0$ and $\epsilon_r>0$, such that, 
\begin{equation*}
    \text{if either $q\in [\epsilon_r,1]$, $q_1=\dots=q_r=1$,} \quad\text{or $q\in (0,\epsilon_r)$, $q_1,\dots,q_r\in [1,1+\delta_r]$,}
\end{equation*} then every off-diagonal entry of $M^{(a_r)}_r$ is positive except $(M_r^{(a_r)})_{r+3,1}=0$.
\end{prop}

\vspace{0.1em}

\begin{proof}[Proof of Proposition \ref{prop-main-simple rightmost}] Let $I_{m}$ be the $m\times m$ identity matrix. By Proposition \ref{prop-PFforprimitive}, it suffices to find $\delta_r,a_r,c_r>0$ such that, when $0<q\le 1$ and $1\le q_i\le 1+\delta_r$, $M^{(a_r)}_r+c_rI_{r+3}$ is nonnegative and $(M^{(a_r)}_r+c_rI_{r+3})^2$ is positive.

For $r=1,2$, take $\delta_r={1\over 2},a_r=1,c_r=2$. Then from Example  \ref{example-r=2}, when $0<q\le 1$ and $1\le q_i\le{3\over 2}$ hold,  sign pattern matrices of $M_1^{(1)}+c_1I_4$ and $M_2^{(1)}+c_2I_5$ are
\begin{equation*}
    \begin{bmatrix}
        + &0& + &+\\
        +&+&0&+\\
        +&+&+&0\\
        0&+&+&+
    \end{bmatrix}\text{ and }\begin{bmatrix}
        +&0&+&+&+\\
        +&+&+&+&+\\
        +&+&+&0&+\\
        +&+&0&+&+\\
        0&+&+&+&+
    \end{bmatrix},
\end{equation*}
which imply that the square of them are positive.
\par

 For $3\le r\le 8$, take $a_r$ as in Proposition \ref{prop-all positive entries but 1}. Since $(M_r^{(a_r)})_{ij}\in\bR[q,q_1,\dots,q_r]$,  there exists $c_r>\!\!>0$ such that, when $q$ and $q_i$ satisfy the conditions in Proposition \ref{prop-all positive entries but 1},  all entries of $M^{(a_r)}_r+cI_{r+3}$ are positive, except for the $(r+3,1)$-entry, which is zero, and this sign pattern implies that $(M_r^{(a_r)}+c_rI_{r+3})^2$ is positive. In particular, this holds when $q\in[\epsilon_r,1]$ and $q_i=1$. Since $(M_r^{(a_r)}+c_rI_{r+3})^2$ being positive is an open condition on $q$ and $q_i$, and $[\epsilon_r,1 ]$ is compact, it follows that we may shrink \(\delta_r\) from Proposition \ref{prop-all positive entries but 1}, if necessary, to ensure that\((M_r^{(a_r)}+c_rI_{r+3})^2\) is positive for $q\in(0,1]$ and $q_i\in[1,1+\delta_r]$. 
\end{proof}

\subsection{ Positivity of off-diagonal entries} Note  $3\le r\le 8$ throughout this subsection but  Lemma \ref{lemma of table} (where the cases $r=1, 2$ are included for completeness).
We prove Proposition \ref{prop-all positive entries but 1} at the end of the subsection; its proof relies on a series of estimates based on the following two key observations for del Pezzo surfaces $X_r$.

\begin{enumerate}[label=(o\arabic*), ref=(o\arabic*)]
\item In the setting of $q_i=1$ and $q\in[0,1]$, there exists $a^*>0$ such that the off-diagonal entries of $M_r^{(a^*)}$ are polynomials in $q$ with \textbf{positive coefficients} (Proposition \ref{proposition-range 1}).\label{observation alpha}

\item For $q$ sufficiently small, estimating \textbf{only the terms of degree $0$ and $1$} in $q$ ensures the necessary positivity, provided $q_i$ remain close to $1$ (Proposition \ref{proposition-range 2}).\label{observation beta}

\end{enumerate}

Denote $(M_r^*)_{ij}:=(M_r^{(a_r^*)})_{ij}$ throughout. 
While the remaining entries are easier to treat as in the proof of Proposition \ref{prop-all positive entries but 1}, the following three types require detailed estimates:

\begin{align*}
    P_{H,j}^{(a_r)}(q,q_i)&:=(M_r^{(a_r)})_{2,j+2}=\sum_{\bfd}\langle c_1,a_rH-\sum E_i, E_j^\vee\rangle^{X_r}_\bfd\, \bfq^{\bfd},\\
    P_{b,j}^{(a_r)}(q,q_i)&:=(M_r^{(a_r)})_{b+2,j+2}=\sum_{\bfd}\langle c_1,E_b, E_j^\vee\rangle^{X_r}_\bfd\, \bfq^{\bfd},\\
    P_{[pt],j}^{(a_r)}(q,q_i)&:=(M_r^{(a_r)})_{r+3,j+2}=\sum_{\bfd}\langle c_1,[pt], E_j^\vee\rangle^{X_r}_\bfd\, \bfq^{\bfd},
\end{align*}
where $1\leq b, j\leq r$ and $ b\neq j$. We prove positivities for $P^*_\bullet:=P^{(a_r^*)}_\bullet$ first at $q_i=1$ in Proposition \ref{proposition-range 1} and then for $q$ small with $q_i$ close to $1$ in Proposition \ref{proposition-range 2}.

To state and prove Proposition \ref{proposition-range 1} and \ref{proposition-range 2}, we need the following Lemma \ref{lemma of table}. Set
 \[
d_{rk}:=\max\{H(\bfd)|\bfd\in\Eff(X_r)\text{ with }\left<\tau_0([pt])^{k-1}\right>^{X_r}_{\bfd}\neq0\},\quad k=1,2,
 \]
 where $H(\bfd)=\int_{\bfd}H\in \mathbb{Z}$.
  Note that $d_{rk}$ is well-defined since 
  there are only finitely many effective curve classes $\bfd$ satisfying $\left<\tau_0([pt])^{k-1}\right>^{X_r}_{\bfd}\neq0$. 
 
  \begin{lemm}\label{lemma of table}
   The values of $d_{r1}$, $d_{r2}$ and $a_r^*:=\frac{rd_{r1}}{3d_{r1}-1}$  are as listed in the following Table. 
\begin{center}
\begin{table}[h]
\centering
    \caption{$d_{r1}$, $d_{r2}$ and $a_r^*$ for each $X_r$}\label{table}
    \begin{tabular}{p{1cm} | p{1cm} p{1cm} p{1cm}  p{1cm} p{1cm} p{1cm} p{1cm} p{1cm} }
  \toprule
  $r$ & 1 & 2 &3 &4 &5 &6 &7 &8\\
  \hline
  $d_{r1}$ & 0& 1 &1 &1 &2 &2&3 &6 \\ 
  
  $d_{r2}$ & 1 &1 &1 &2 &2 &3 &5 &11  \\
 
  $a_r^*$ & 0 &1 &3/2 & 2 &2 &12/5 & 21/8 & 48/17\\  
  \bottomrule
\end{tabular}
\vspace{-3em}
\end{table}
\end{center}
  \end{lemm}
  \begin{proof}
    By \cite[Section 5.2]{GP98}, any contributing curve class $\bfd$  must have non-negative arithmetic genus, i.e. $(H(\bfd)-1)(H(\bfd)-2)-\suml_{i=1}^rE_i(\bfd)(E_i(\bfd)-1)\ge0$. Using dimension constraint for equality and  Cauchy-Schwarz for inequality, we have
\begin{equation*}
    (3H(\bfd)-k)+\sum_{i=1}^rE_i(\bfd)(E_i(\bfd)-1)=\sum_{i=1}^rE_i(\bfd)^2\geq{1\over r}(\sum_{i=1}^rE_i(\bfd))^2={1\over r}\left(3H(\bfd)-k\right)^2.
\end{equation*}
Combining the two gives $\frac{1}{r}(3H(\bfd)-k)^2-(3H(\bfd)-k)
\le (H(\bfd)-1)(H(\bfd)-2)$. This yields the quadratic constraint in $d_{rk}$ 
   \[
   9d_{rk}^2-6d_{rk}k+k^2-r(3d_{rk}-k)\leq rd_{rk}^2-3rd_{rk}+2r.
   \]
For $k=2$, the constraint gives  
\[
d_{r2}\le {6+2\sqrt{r}\over9-r}\Longrightarrow d_{r2}\le \lfloor{6+2\sqrt{r}\over9-r}\rfloor\quad(\text{use }d_{r2}\in\mathbb{Z}_{\ge0}).
\] 
One can check that $\lfloor{6+2\sqrt{r}\over9-r}\rfloor$ is exactly the number listed in the third row for $r=1,\dots,8$ (e.g. $\lfloor{6+2\sqrt{r}\over9-r}\rfloor=11$ when $r=8$). To prove $d_{r2}=\lfloor{6+2\sqrt{r}\over9-r}\rfloor$, it suffices to find $\bfd\in\Eff(X_r)$ satisfying $\left<[pt]\right>^{X_r}_{\bfd_r}\neq0$ and $H(\bfd)=\lfloor{6+2\sqrt{r}\over9-r}\rfloor$. Note that $\left<[pt]\right>^{X_r}_{H-E_1}\neq0$, and we can apply Cremona transformation (see \cite[Section 5.1]{GP98}) repeatedly to $H-E_1$ to find such $\bfd\in\Eff(X_r)$. For instance $\bfd=11H-\suml_{i=1}^74E_i-3E_8$ when $r=8$. This proves the case $k=2$. The case $k=1$ is similar.
  % \textcolor{red}{The proof needs more details?}
  \end{proof}

 For $d,N\in\bZ_{\ge1}$,  let $\Eff_{d,N}(X_r)$ (resp.  $\Eff_{d,N}^{[pt]}(X_r)$) denote the subset of $\Eff(X_r)$ consisting of effective classes $\bfd$ satisfying $H(\bfd)=d$ and $\<\>^{X_r}_{\bfd}=N$ (resp. $\<[pt]\>^{X_r}_{\bfd}=N$).\par
 The following lemma relies on the symmetry of the exceptional divisors $E_i$ and plays a central role throughout our subsequent arguments.

\begin{lemm}\label{lemma-symmetrizedsum}
For $1\le b\le r$, we have $\sum_{\bfd\in\Eff_{d,N}(X_r)}E_b(\bfd)=\sum_{\bfd\in\Eff_{d,N}(X_r)}{3d-1\over r}$ and $\sum_{\bfd\in\Eff_{d,N}^{[pt]}(X_r)}E_b(\bfd)=\sum_{\bfd\in\Eff_{d,N}^{[pt]}(X_r)}{3d-2\over r}$
\end{lemm}
\begin{proof}
    For $\bfd\in \Eff_{d,N}(X_r)$, we have $c_1(\bfd)=3H(\bfd)-\suml_{i=1}^rE_i(\bfd)=1$ by the dimension constraint. Moreover, by   symmetry among divisors $E_i$'s (see \cite[(P3) before Remark 3.4]{GP98}), for each $b$ the quantity $\sum_{\bfd\in\Eff_{d,N}(X_r)}E_b(\bfd)$ is the same. 
 Therefore
    \begin{align*}
      \sum_{\bfd\in\Eff_{d,N}(X_r)}(3d-1)=  \sum\limits_{\bfd\in\Eff_{d,N}(X_r)}\sum\limits_{i=1}^rE_i(\bfd)=\sum_{i=1}^r\sum_{\bfd\in\Eff_{d,N}(X_r)}E_i(\bfd)=r\sum_{\bfd\in\Eff_{d,N}(X_r)}E_b(\bfd),
    \end{align*} 
    where we exchanged the order of summation since $\Eff_{d,N}(X_r)$ is a finite set. This proves the first equality. The proof for the second equality is similar.
\end{proof}

We proceed by determining conditions on $a_r$ that ensure $P_\bullet^{(a_r)}|_{q_i=1}\geq0$.
\begin{lemm}
\label{proposition-range 1 P1}
Let $3\le r\le 8$. Set $\Lambda_1(d):=\min\{\frac{3d-1}{d},\frac{rd}{3d-1}\}$ and $\Lambda_2(d):=\max\{\frac{3d-1}{d},\frac{rd}{3d-1}\}$ for $1\le d\le d_{r1}$. If
\begin{equation*}
a_r\in \bigcap_{1\leq d\leq d_{r1}}[\Lambda_1(d),\Lambda_2(d)],
\end{equation*}
then for any $q\in(0,1]$, we have $P_{H,j}^{(a_r)}(q,1)>0$ with $1\leq j\leq r$.
\end{lemm}

\begin{proof}
  The divisor axiom, the dimension constraint, and Lemma \ref{lemma-symmetrizedsum} give
\begin{align*}
P_{H,j}^{(a_r)}(q,1)={}&\sum_{\bfd}\langle \rangle^{X_r}_\bfd\left((a_r-3)H(\bfd)+1\right)\left(\frac{H(\bfd)}{a_r}-E_j(\bfd)\right)q^{H(\bfd)}\\%\quad(\text{dimension constraint})\notag\\
        ={}&\sum_{\bfd}\langle \rangle^{X_r}_\bfd(a_rH(\bfd)-3H(\bfd)+1)\frac{1}{r}\left(\frac{rH(\bfd)}{a_r}-3H(\bfd)+1\right)q^{H(\bfd)}\text{ ( Lemma \ref{lemma-symmetrizedsum})}.
%&=\sum_{H(\bfd)=0}\langle  \rangle^{X_r}_{\bfd}\frac{1}{r}+\sum_{ H(\bfd)\geq 1 }\langle \rangle^{X_r}_\bfd(a_rH(\bfd)-3H(\bfd)+1)\frac{1}{r}\left(\frac{rH(\bfd)}{a_r}-3H(\bfd)+1\right)q^{H(\bfd)}.\notag
\end{align*}
For $H(\bfd)=d$ and $1\leq d\leq d_{r1}$, the assumed range of $a_r$ ensures the nonnegativity of the degree-$d$ coefficient of $P_{H,j}^{(a_r)}(q,1)$: $(a_rd-3d+1)(\frac{rd}{a_r}-3d+1)\geq 0$. Therefore
\begin{equation*}
    P_{H,j}^{(a_r)}(q,1)\ge \sum_{H(\bfd)=0}\langle  \rangle^{X_r}_{\bfd}\frac{1}{r}=\sum_{1\leq i\leq r}\langle  \rangle^{X_r}_{E_i}\frac{1}{r}=1>0.   \qedhere
\end{equation*}\end{proof}

\begin{lemm}\label{proposition-range 1 P2}
Let $3\le r\le 8$ and $a_r\in (0,\frac{rd_{r1}}{3d_{r1}-1}]$. If $q\in(0,1]$ and $q_i=1$, then we have  
\begin{equation*}
        P_{b,j}^{(a_r)}(q,1)> 0,
    \quad  1\leq b, j\leq r\text{ with } b\neq j.
\end{equation*}
\end{lemm}

\begin{proof}
Let $b\neq j$. By divisor axiom and dimension constraint, $P_{b,j}^{(a_r)}(q,1)$ is equal to
       $\sum_{\bfd}\langle \rangle^{X_r}_{\bfd}E_b(\bfd)(\frac{H(\bfd)}{a_r}-E_j(\bfd))q^{H(\bfd)}.$
  Note that the terms with $H(\bfd)=0$ do not contribute to the sum, since then $\bfd=E_i$ and $b\ne j$ implies $E_b(\bfd)(\frac{H(\bfd)}{a_r}-E_j(\bfd))=0$. Hence
   \[
   P_{b,j}^{(a_r)}(q,1)=\suml_{d,N\ge1}\suml_{\bfd\in\Eff_{d,N}(X_r)}N\left(E_b(\bfd){d\over a_r}-E_b(\bfd)E_j(\bfd)\right)q^d.
   \]
   
   For $\Eff_{d,N}(X_r)\ne\varnothing$, set $\sigma:=\left(\sum_{\bfd\in\Eff_{d,N}(X_r)}\left(E_b(\bfd)E_j(\bfd)\right)\right)\cdot|\Eff_{d,N}(X_r)|^{-1}$, which is the average over summation and independent of $b,j$ by symmetry among divisors $E_i$'s. Using $\sum_{\bfd\in\Eff_{d,N}(X_r)}\sigma=\sum_{\bfd\in\Eff_{d,N}(X_r)}E_b(\bfd)E_j(\bfd)$, we get
\begin{align}
       \sum_{\bfd\in\Eff_{d,N}(X_r)} (3d-1)^2 &=\sum_{\bfd\in\Eff_{d,N}(X_r)}\left(\sum_{i=1}^r E_i(\bfd)\right)^2\quad(\text{use } c_1(\bfd)=1)\notag\\
       &=\sum_{\bfd\in\Eff_{d,N}(X_r)}\left(\sum_{i=1}^rE_i(\bfd)^2+r(r-1)\sigma\right)\notag\\
       &\geq \sum_{\bfd\in\Eff_{d,N}(X_r)} \left(\frac{(3d-1)^2}{r}+r(r-1)\sigma \right). \label{N=1} 
\end{align}
In the last inequality we have used $c_1(\bfd)=1$ and the Cauchy-Schwarz inequality, which states that $ r(\sum_{i=1}^rE_i(\bfd)^2)\geq(\sum_{i=1}^rE_i(\bfd))^2=(3d-1)^2$, with ``$=$" holding if and only if $E_1(\bfd)=\cdots=E_r(\bfd)$.
Therefore we can conclude that $(\frac{3d-1}{r})^2\geq \sigma$. Hence
\begin{equation*}
    \sum_{\bfd\in\Eff_{d,N}(X_r)} E_b(\bfd)\frac{3d-1}{r}\overset{\text{Lemma }\ref{lemma-symmetrizedsum}}{=}\sum_{\bfd\in\Eff_{d,N}(X_r)} \left(\frac{3d-1}{r}\right)^2\geq\sum_{\bfd\in\Eff_{d,N}(X_r)}E_b(\bfd)E_j(\bfd).
\end{equation*}
Then we obtain the following inequality on the degree-$d$ coefficient of $P_{b,j}^{(a_r)}(q,1)$:
\begin{align*}
   \sum_{\bfd\in\Eff_{d,N}(X_r)}E_b(\bfd)\frac{d}{a_r}-E_b(\bfd)E_j(\bfd)\geq\sum_{\bfd\in\Eff_{d,N}(X_r)}\frac{E_b(\bfd)}{r}\frac{rd}{a_r}-3d+1.
\end{align*}
Since $rd/(3d-1)$ is decreasing about $d$ and $a_r\le rd_{r1}/(3d_{r1}-1)$, the coefficients are nonnegative for $1\le d\le d_{r1}$, then so is $P_{b,j}^{(a_r)}(q,1)$. Strict positivity of $P_{b,j}^{(a_r)}(q,1)$ follows from  that $\langle  \rangle^{X_r}_{H-E_1-E_2}=1\neq 0$  and the Cauchy--Schwarz inequality in \eqref{N=1} is strict for this $H-E_1-E_2$ when $r\geq 3$. \end{proof}
\begin{lemm}\label{proposition-range 1 P3}
Let $3\le r\le 8$, and set
\begin{equation}\label{domain 3}
D:=\begin{cases}
        (0,3),&r=3,\\
         (0,\tfrac{rd_{r2}}{3d_{r2}-2}],&r\ge4.\end{cases}
\end{equation}
For $a_r\in D$, if $q\in(0,1]$, then $P_{[pt],j}^{(a_r)}(q,1)>0$ with $1\leq  j\leq r$.
\end{lemm}

\begin{proof} Observe that, by \eqref{eq--invariants}, $H-E_1\in\Eff_{1,1}^{[pt]}(X_r)$, so the sum over $Eff_{d,N}^{[pt]}(X_r)$ is nonempty.
      By divisor axiom, dimension constraint, and Lemma \ref{lemma-symmetrizedsum}, $P_{[pt],j}^{(a_r)}(q,1)$ is equal to
$\sum_{\bfd}2\langle[pt]\rangle^{X_r}_{\bfd}\frac{1}{r}\left(\frac{rH(\bfd)}{a_r}-3H(\bfd)+2\right)q^{H(\bfd)}.$
Recall that $\langle[pt]\rangle^{X_r}_{\bfd}\in\bZ_{\ge0}$, and $\langle[pt]\rangle^{X_r}_{\bfd}\neq 0$ implies $H(\bfd)\geq 1$. So we have
\[
P^{a_r}_{[pt],j}(q,1)=\suml_{d,N\ge1}\suml_{\bfd\in\Eff_{d,N}^{[pt]}(X_r)}2N{1\over r}\left({rd\over a_r}-3d+2\right)q^d.
\]

For $r=3$, \ref{table} gives $d_{r2}=1$, hence the factor ${rd\over a_r}-3d+2$ reduces to ${1\over a_3}-{1\over 3}$. Since $a_3\in (0,3)$ by \eqref{domain 3}, we obtain $
P_{[pt],j}^{(a_3)}(q,1)>0$.\par

For $r\ge 4$, Table \ref{table} gives $d_{r2}\ge 2$.  Since
$\frac{rd}{3d-2}=\frac r3+\frac{2r}{3(3d-2)}$is strictly decreasing in $d$, \eqref{domain 3}
implies the coefficients of $P_{[pt],j}^{(a_3)}(q,1)$ are nonnegative, and
\[
\frac{rd}{a_r}-3d+2>0 \quad (1\le d\le d_{r2}-1),
\qquad
\frac{rd}{a_r}-3d+2\ge 0 \quad (d=d_{r2}).
\]
Moreover,  $H-E_1\in\Eff_{1,1}^{[pt]}(X_r)$, so its contribution is strictly positive, while all remaining contributions with $1\le d\le d_{r2}$ are nonnegative.  Therefore $P_{[pt],j}^{(a_r)}(q,1)>0$ for $r\ge 4$, this proves the lemma.
\end{proof}

\begin{prop}\label{proposition-range 1}Let $3\le r\le 8$.  Then for any $q\in (0,1]$, we have
\begin{equation*}
    P_{H,j}^*(q,1)>0,\quad P_{[pt],j}^*(q,1)>0,\quad
    P_{b,j}^*(q,1)>0,\quad 1\le b,j\le r\text{ with }b\ne j.
\end{equation*}

\end{prop}
\begin{proof} By Lemmata \ref{proposition-range 1 P1}, \ref{proposition-range 1 P2}, and \ref{proposition-range 1 P3}, and with the notation therein, it remains to show that 
    \begin{align*}
        a_r^*\in\bigcap_{1\leq d\leq d_{r1}}[\Lambda_1(d),\Lambda_2(d)]\cap (0,\frac{rd_{r1}}{3d_{r1}-1}]\cap D.
    \end{align*}

Recall that  $\frac{rd}{3d-2}$ is strictly decreasing in $d\geq 1$, and so is  $\frac{rd}{3d-1}$. Checking with Table \ref{table}, we have $\frac{rd_{r1}}{3d_{r1}-1}<\frac{rd_{r2}}{3d_{r2}-2}$ for $r=3$, and $\frac{rd_{r1}}{3d_{r1}-1}\leq\frac{rd_{r2}}{3d_{r2}-2}$ for $r\geq 4$, which implies that
\begin{align*}
    a_r^*=\frac{rd_{r1}}{3d_{r1}-1}\in(0,\frac{rd_{r1}}{3d_{r1}-1}]\cap D.
\end{align*}

It remains to show that $a_r^*\in\bigcap_{1\leq d\leq d_{r1}}[\Lambda_1(d),\Lambda_2(d)]$.
For $r=3,4$, Table~\ref{table} gives $d_{r1}=1$, so that $a_r = \tfrac{rd_{r1}}{3d_{r1}-1} \in \bigcap_{1\leq d\leq d_{r1}} [\Lambda_1(d),\Lambda_2(d)]$
is automatically satisfied. For $r\ge 5$, Table~\ref{table} gives $d_{r1}\geq 2$, and moreover one can check case by case that
\[
  \frac{3(d_{r1}-1)-1} {d_{r1}-1}\leq \frac{rd_{r1}}{3d_{r1}-1}=a_r^*.
\]
Notice that $\tfrac{3d-1}{d}$ is strictly increasing in  $d\geq 1$. Therefore for all $1 \leq d \leq d_{r1}-1$,
\begin{align*}
    \frac{3d-1}{d}
    \;\leq\; \frac{3(d_{r1}-1)-1}{d_{r1}-1}
    \;\leq\; a = \frac{rd_{r1}}{3d_{r1}-1}
    \;<\; \frac{rd}{3d-1}.
\end{align*}
 Moreover  $a_r^*=\tfrac{rd_{r1}}{3d_{r1}-1}$ is one of $\Lambda_1(d_{r1}), \Lambda_2(d_{r1})$,    thus $a_r^* 
   \in \bigcap_{1\leq d\leq d_{r1}} [\Lambda_1(d),\Lambda_2(d)]$
as required, which concludes the proof.
\end{proof}

As noted at the beginning of this subsection, the sign of $P^*_\bullet(q,q_i)$ for sufficiently small $q>0$ is  governed by its leading term. This observation yields the following result.

\begin{prop}\label{proposition-range 2}
    Let $3\le r\le 8$. There exist  positive numbers $\delta_r\in (0,\frac{1}{2}]$ and $\epsilon_r$,  such that, for any $q\in(0,\epsilon_r]$ and $q_i\in[1,1+\delta_r]$, we have
\begin{equation*}
    P_{H,j}^{*}(q,q_i)>0,\quad 
    P_{[pt],j}^{*}(q,q_i)>0,\quad P_{b,j}^{*}(q,q_i)>0,\quad 1\le b,j\le r \text{ with } b\neq j.
\end{equation*} 
 
\end{prop}

 \begin{proof}  For $\bullet=(H,j),(b,j),([pt],j)$, write $P^*_\bullet$ as the sum of its leading term and higher order  terms in $q$, namely
     \begin{equation*}
         P^*_\bullet(q,q_i)=q^k\cdot L_\bullet(q_i)+q^{k+1}\cdot R_\bullet (q,q_i),
     \end{equation*}
     where  $k\in\mathbb{Z}_{\geq 0}$, $L_\bullet$ is a nonzero polynomial in $q_i$, and $R_{\bullet}$ is a polynomial in $q,q_i$.  If $L_\bullet$ has a positive lower bound on a box $[1,1+\delta]^r$, then boundedness of $R_\bullet$ on $0\le q\le1$, $1\le q_i\le\frac32$ gives $P_\bullet^*>0$ on $(0,\epsilon]\times[1,1+\delta]^r$ for sufficiently small $\epsilon$. \par
Now  it remains to prove that there exists  $\delta_r\in(0,\frac{1}{2}]$ such that $L_\bullet(q_i)$ has a positive lower bound when $(q_i)\in[1,1+\delta_r]^r$.

     For $L_{H,j}$ with $r\geq 3$, note that $P_{H,j}^{*}(q,q_i)=\langle \rangle_{E_j}^{X_r}  q_j\cdot q^0+R_{H,j}(q,q_i)q^1$.   With $\langle \rangle_{E_j}^{X_r} $, we have $L_{H,j}(q_i)=q_j\geq 1$ for any  $(q_i)\in [1,1+\frac{1}{2}]^r$.\par

     For $L_{b,j}$ with $r\geq 3$, we have
     \begin{equation*}
         P_{b,j}^{*}(q,q_i)=\sum_{\substack{1\le k\le r\\k\ne b}}\langle c_1,E_b,E_j^\vee
         \rangle_{H-E_b-E_k}^{X_r}\left(\prod_{i\neq b,k} q_k\right)q+\sum_{\bfd:H(\bfd)\geq 2} \langle c_1,E_b,E_j^\vee
         \rangle^{X_r}_{\bfd}\mathbf{q}^\bfd .   
     \end{equation*}
     So together with $\langle \rangle^{X_r}_{H-E_b-E_k}=1$, the divisor equation gives
     \[
     L_{b,j}(q_i)=(\frac{1}{a_r^*}-1)\prod_{i\neq b,j} q_i+\sum_{k\neq b,j}\frac{1}{a_r^*} \prod_{i\neq b,k} q_i.
     \]
  By Table \ref{table}, for $r\ge 3$ one has $1<a_r^*=\frac{rd_{r1}}{3d_{r1}-1}<r-1$, and hence $\frac{1}{a_r^*}-1+\frac{r-2}{a_r^*}>0$. Therefore,  for all $(q_i)\in[1,1+\delta_{r1}^*]^r$with choosing $\delta_{r1}^*>0$ sufficiently small, we have 
\begin{equation*}
  L_{b,j}(q_i)\geq (\frac{1}{a_r^*}-1)(1+\delta_{r1})^{r-2}+\sum_{k\neq b,j}\frac{1}{a_r^*}=(\frac{1}{a_r^*}-1)(1+\delta_{r1})^{r-2}+\frac{r-2}{a_r^*}>0.
\end{equation*}
Thus $L_{b,j}(q_i)$ has a positive lower bound on $[1,1+\delta_{r1}^*]^r$.

     For $L_{[pt],j}$ with $r\geq 3$, 
     we have
     \begin{equation*}
         P^*_{[pt],j}=\sum_{1\leq k \leq r}\langle c_1, [pt],E_j^\vee\rangle_{H-E_k}^{X_r}\left(\prod_{i\neq k}q_i\right)q+\sum_{\bfd:H(\bfd)\geq 2}\langle c_1,[pt],E_j^\vee\rangle_\bfd^{X_r}\mathbf{q}^\bfd.
     \end{equation*}
     So together with $\langle[pt]\rangle_{H-E_k}=1$, the divisor equation gives
     \[
     L_{[pt],j}=2(\frac{1}{a_r^*}-1)\prod_{i\neq j}q_i+
             \sum_{k\neq j}2\frac{1}{a_r^*}\prod_{i\neq k}q_i.
     \]
By Table \ref{table}, for $r\ge 3$ one has $1<a_r^*<r-1<r$,
and hence $\left(\frac{1}{a_r^*}-1\right)+\frac{r-1}{a_r^*}=\frac{r-a_r^*}{a_r^*}>0$.
Therefore, for all $(q_i)\in[1,1+\delta_{r2}^*]^r$ with choosing $\delta_{r2}^*>0$ sufficiently small, we have
\begin{equation*}    
L_{[pt],j}(q_i)\ge2\left(\frac{1}{a_r^*}-1\right)(1+\delta_{r2}^*)^{r-1}
+\sum_{k\ne j}\frac{2}{a_r^*} =2\left(\frac{1}{a_r^*}-1\right)(1+\delta_{r2}^*)^{r-1}+\frac{2(r-1)}{a_r^*}>0.
\end{equation*}
Thus $L_{[pt],j}(q_i)$ has a positive lower bound on $[1,1+\delta_{r2}^*]^r$.
     
Finally,   taking $\delta_r=\min\{\frac{1}{2},\delta_{r1}^*,\delta_{r2}^*\}$ completes the proof.
 \end{proof}

Now we end this subsection with the proof of Proposition \ref{prop-all positive entries but 1}.
\begin{proof}[Proof of Proposition \ref{prop-all positive entries but 1}] Take $a_r = a_r^*$ and denote $(M_r^*)_{ij} := (M_r^{(a_r^*)})_{ij}$.  
By Propositions \ref{proposition-range 1}
and \ref{proposition-range 2},
the required positivity of $P_\bullet^*(q,q_i)$ for $r\geq 3$ is ensured, and it remains to consider the following four cases. We use $T$ to denote an arbitrary class in the basis $[\mathbf{1},a_r^*H-\suml_{i=1}^rE_i,E_1,...,E_r,[pt]]$.\par

\begin{enumerate}[label=(\arabic*), leftmargin=0pt, itemindent=2.pt, labelsep=0.5em]
    \item[(1)] The  \textbf{first column} is given by $\langle c_1,\mathbf{1},T^\vee\rangle_\bfd$. We have
 \begin{equation*}
 (M_r^*)_{11}=0,\quad (M_r^*)_{21}=\frac{3}{a_r^*},\quad (M_r^*)_{31}=\cdots=(M_r^*)_{r+2,1}=\frac{3}{a_r^*}-1,\quad (M_r^*)_{r+3,1}= 0. 
 \end{equation*}  
 By Table \ref{table}, we see  $(M_r^*)_{j1}>0$ for $j=2,\cdots,r+2$.

\item[(2)] The \textbf{last row} is given by  $\langle c_1,T, [pt]^\vee\rangle_\bfd^{X_r}$. We have
 \begin{equation*} \label{condition a2}
     (M_r^*)_{r+3,1}=0,\quad (M_r^*)_{ r+3,2}=3a_r^*-r,\quad (M_r^*)_{r+3,3}=\cdots =(M_r^*)_{r+3,r+2}=1,\quad (M_r^*)_{r+3,r+3}=0.
 \end{equation*}\par
By Table \ref{table}, we see $(M_r^*)_{r+3,2}>0$.

\item [(3)] 
The \textbf{first row} is given by $ \langle c_1,T, \mathbf{1}^\vee \rangle_\bfd^{X_r}$. The dimension constraint and divisor axiom give\begin{align*}
(M_r^*)_{12}&=\sum_{\bfd}\langle c_1,a_r^*H-\sum E_i,\mathbf{1}^\vee\rangle_\bfd^{X_r}\mathbf{q}^\bfd =\sum_{\bfd}2\left((a_r^*-3)H(\bfd)+2\right)\langle [pt]\rangle_\bfd^{X_r}\mathbf{q}^\bfd,\\
 (M_r^*)_{1,b+2}&=\sum_{\bfd}\langle c_1,E_b,\mathbf{1}^\vee\rangle_\bfd^{X_r}\mathbf{q}^\bfd =\sum_{\bfd}2E_b(\bfd) \langle [pt]\rangle_\bfd^{X_r}\mathbf{q}^\bfd \quad \text{for any }1\leq b\leq r,\\
 (M_r^*)_{1,r+3}&=\sum_{\bfd}\langle c_1,[pt],\mathbf{1}^\vee\rangle_\bfd^{X_r}\mathbf{q}^\bfd =\sum_{\bfd}3 \langle [pt]^2\rangle_\bfd^{X_r}\mathbf{q}^\bfd.
\end{align*}

For $(M_r^*)_{12}$, note that $a_r^*<3$, and then for $1\le d\le d_{r2}$, we have 
\begin{equation*}
    (a_r^*-3)d+2\geq (a_r^*-3)d_{r2}+2\overset{\text{Table } \ref{table}}{=}\begin{cases}
0,&   r=4,5,\\
\frac{1}{3d_{r1}-1}>0,&r=3,6,7,8,
    \end{cases}
\end{equation*}
implying that $\left((a_r^*-3)H(\bfd)+2\right)\langle[pt]\rangle_\bfd^{X_r}\ge0$ for all $\bfd$. In particular,  considering $\bfd=H-E_1$ for $r\ge3$, we get $(M_r^*)_{12}>0$.  Note that $E_b(\bfd)\langle[pt]\rangle_\bfd^{X_r}\ge0$ for all $\bfd$,  considering $\bfd=H-E_b$, we get $(M_r^*)_{1,b+2}>0$. Finally, $\langle[pt]^2\rangle_H^{X_r}>0$ gives $(M_r^*)_{1,r+3}>0$. 

\item[(4)]  The \textbf{second row} is given by $\langle c_1,T,(aH-\sum E_i)^\vee\rangle_\bfd^{X_r}$, the same argument gives
  \begin{equation*}
        \begin{aligned}
 (M_r^*)_{2,b+2}&=\sum_{\bfd}\langle c_1,E_b,(a_r^*H-\sum E_i)^\vee\rangle_\bfd^{X_r}\mathbf{q}^\bfd =\sum_{\bfd}E_b(\bfd) \frac{H(\bfd)}{a_r^*} \langle  \rangle_\bfd^{X_r}\mathbf{q}^\bfd \quad \text{for any }1\leq b\leq r,\\
 (M_r^*)_{2,r+3}&=\sum_{\bfd}\langle c_1,[pt],(a_r^*H-\sum E_i)^\vee\rangle_\bfd^{X_r}\mathbf{q}^\bfd =\sum_{\bfd}2\frac{H(\bfd)}{a_r^*} \langle [pt]\rangle_\bfd^{X_r}\mathbf{q}^\bfd.
        \end{aligned}
    \end{equation*} Note that  $E_b(\bfd)H(\bfd)\langle \rangle_\bfd^{X_r}\geq0$  
for all $\bfd$, considering $\bfd=H-E_b-E_1 $, we get   $(M_r^*)_{2,b+2}>0$.
Similarly, considering $\bfd=H-E_1 $, we get   $(M_r^*)_{2,r+3}>0$.\qedhere
\end{enumerate}
 \end{proof}

\bibliographystyle{abbrv}

\end{document}